\documentclass[11pt,a4paper,leqno]{amsart}

\usepackage[latin1]{inputenc}
\usepackage[T1]{fontenc}
\usepackage{amsfonts}
\usepackage{amsmath}
\usepackage{amssymb}
\usepackage{eurosym}
\usepackage{mathrsfs}
\usepackage{palatino}
\usepackage{color}
\usepackage{esint}
\usepackage{url}
\usepackage{verbatim}
\usepackage{enumerate}
\usepackage[dvipsnames]{xcolor}

\usepackage{caption,float}
\usepackage{placeins}
\usepackage{afterpage}

\usepackage{subfig}

\usepackage[pagebackref,hypertexnames=false, colorlinks, citecolor=RoyalPurple, linkcolor=RoyalPurple, urlcolor=RoyalPurple]{hyperref}
\usepackage{amsrefs}
\newfloat{fig}{tb}{lop}[section]
\DeclareCaptionLabelFormat{Ucase}{Figure~#2}
\newcommand{\R}{\mathbb{R}}
\newcommand{\C}{\mathbb{C}}

\newcommand{\N}{\mathbb{N}}

\newcommand{\scrS}{\mathscr{S}}

\newcommand{\calK}{\mathcal{K}}

\newcommand{\calQ}{\mathcal{Q}}

\newcommand{\calR}{\mathcal{R}}

\newcommand{\calO}{\mathcal{O}}

\newcommand{\scrD}{\mathscr{D}}

\newcommand{\diam}{\operatorname{diam}}

\newcommand{\SIO}{\textup{SIO}}

\newcommand{\CZO}{\textup{CZO}}

\numberwithin{equation}{section}

\newcommand{\ud}[0]{\,\mathrm{d}}

\newcommand{\dist}[0]{\operatorname{dist}}

\newcommand{\op}[1]{\operatorname{#1}}

\newcommand{\abs}[1]{|#1|}
\newcommand{\babs}[1]{\big|#1\big|}

\newcommand{\Norm}[2]{\|#1\|_{#2}}
\newcommand{\bNorm}[2]{\big\|#1\big\|_{#2}}
\newcommand{\BNorm}[2]{\Big\|#1\Big\|_{#2}}

\newcommand{\ave}[1]{\langle #1\rangle}
\newcommand{\bave}[1]{\big\langle #1\big\rangle}
\newcommand{\Bave}[1]{\Big\langle #1\Big\rangle}

\newcommand{\script}[1]{\mathscr{#1}}

\newcommand{\BMO}[0]{\operatorname{BMO}}

\newcommand{\VMO}[0]{\operatorname{VMO}}

\newcommand{\supp}[0]{\operatorname{spt}}
\newcommand{\loc}[0]{\operatorname{loc}}

\newcommand{\ch}[0]{\operatorname{ch}}

\newcommand{\wt}[1]{{\widetilde{#1}}}

\swapnumbers
\theoremstyle{plain}
\newtheorem{thm}[equation]{Theorem}
\newtheorem{lem}[equation]{Lemma}
\newtheorem{prop}[equation]{Proposition}
\newtheorem{cor}[equation]{Corollary}

\theoremstyle{definition}
\newtheorem{defn}[equation]{Definition}

\theoremstyle{remark}

\newtheorem{rem}[equation]{Remark}

 \title{Fractional Bloom boundedness and compactness of commutators}

\author{Tuomas Hyt\"{o}nen \and Tuomas Oikari \and Jaakko Sinko}

\address{Department of Mathematics and Statistics, P.O.B. 68 (Pietari Kalmin katu 5), FI-00014 University of Helsinki, Finland}
\email{\{tuomas.hytonen,tuomas.v.oikari,jaakko.sinko\}@helsinki.fi}

\makeatletter
\@namedef{subjclassname@2020}{%
	\textup{2020} Mathematics Subject Classification}
\makeatother

\thanks{The authors were supported by the Academy of Finland through project Nos.\ 314829 (all), 336323 (Sinko) and 346314 (Hyt\"onen), as well as by the Jenny and Antti Wihuri Foundation (Oikari).}

\begin{document}
\begin{abstract} Let $T$ be a non-degenerate Calder\'on-Zygmund operator 
	and let $b:\R^d\to\C$ be locally integrable.  Let $1<p\leq q<\infty$ and let $\mu^p\in A_p$ and $\lambda^q\in A_q,$ where $A_{p}$ denotes the usual class of Muckenhoupt weights. 
We
 show that 
	\begin{align*}
		\Norm{[b,T]}{L^p_{\mu}\to L^q_{\lambda}}\sim \Norm{b}{\BMO_{\nu}^{\alpha}},\qquad 	[b,T]\in \calK(L^p_{\mu}, L^q_{\lambda})\quad\mbox{iff}\quad b\in \VMO_{\nu}^{\alpha},
	\end{align*}
where $L^p_\mu=L^p(\mu^p)$ and $\alpha/d = 1/p-1/q$, the symbol $\calK$ stands for the class of compact operators between the given spaces, and the fractional weighted $\BMO_{\nu}^{\alpha}$ and $\VMO_{\nu}^{\alpha}$ spaces are defined through the following fractional oscillation and Bloom weight
	\begin{align*}
		\calO_{\nu}^{\alpha}(b;Q) = \nu(Q)^{-\alpha/d}\Big(\frac{1}{\nu(Q)}\int_Q\abs{b-\ave{b}_Q}\Big),\qquad \nu 
		=  \big(\frac{\mu}{\lambda}\big)^{\beta},\quad \beta = (1+\alpha/d)^{-1}.
	\end{align*}
The key novelty is dealing with the off-diagonal range $p<q$, whereas the case $p=q$ was previously studied by Lacey and Li. However, another novelty in both cases is that our approach allows complex-valued functions $b$, while other arguments based on the median of $b$ on a set are inherently real-valued.
\end{abstract}

\maketitle
\section{Introduction and background}


\subsection{Unweighted commutator theory}
The study of commutators of singular integrals have their roots in the work of Nehari \cite{Nehari1957}, where the boundedness of the commutator of the Hilbert transform and a multiplication by $b$ (the symbol of the commutator)
was characterized through a connection with Hankel operators. Later, in 1976, Coifman, Rochberg and Weiss \cite{CRW} developed real-analytic methods and extended Nehari's result by providing the following commutator lower- and upper bounds \footnote{For notation see Section \ref{sect:notation} below.}
\begin{equation}\label{eq:commutatorbmo}
	\|b\|_{\BMO(\R^d)} \lesssim \sum_{j=1}^d\|[b,\calR_j]\|_{L^p(\R^d) \to L^p(\R^d)} \lesssim \|b\|_{\BMO(\R^d)},
\end{equation}
where $p\in(1,\infty),$ $\BMO = \BMO^0,$ and we denote
\begin{align}\label{defn:BMO}
	\|b\|_{\BMO^{\alpha}} :=\sup_{Q\in\calQ} \calO^{\alpha}(b;Q),\qquad \calO^{\alpha}(b;Q) = \ell(Q)^{-\alpha}\fint_Q  \abs{b-\ave{b}_Q},
\end{align}
where $\calQ$ stands for the collection of all cubes. When $\alpha = 0,$ we drop the superscript. The upper bound \eqref{eq:commutatorbmo} was proved in \cite{CRW} for a wide class of bounded singular integrals with convolution kernels, while the lower bound especially concerned the Riesz transforms
\begin{align*}
	\calR_jf(x) = p.v.\int_{\R^d}\frac{x_j-y_j}{\abs{x-y}^{d+1}}f(y)\ud y.
\end{align*} 
The lower bound in \eqref{eq:commutatorbmo} was improved separately by both Janson \cite{Janson1978} and Uchiyama \cite{Uch1978} to $\|b\|_{\BMO} \lesssim \Norm{[b,T]}{{L^p(\R^d) \to L^p(\R^d)}}$ for a wider class of singular integrals that includes any single Riesz transform (in contrast with \eqref{eq:commutatorbmo} involving all the $d$ Riesz transforms). Janson's proof also gives the following off-diagonal extension of the boundedness of the commutator in terms of the homogeneous Hölder space; when $1<p<q<\infty,$ there holds that
\begin{align}\label{eq:commutatoralpha}
	\Norm{[b,T]}{L^p(\R^d)\to L^q(\R^d)}\sim \Norm{b}{\dot{C}^{0,\alpha}(\R^d)}:=\sup_{x\not=y}\frac{\abs{b(x)-b(y)}}{\abs{x-y}^{\alpha}}, \qquad \frac{\alpha}{d} = \frac{1}{p}-\frac{1}{q}.
\end{align}
When $\alpha>0,$ an elementary argument shows that $\Norm{b}{\dot{C}^{0,\alpha}}\sim 	\|b\|_{\BMO^{\alpha}},$ and
hence, for $1<p\leq q<\infty,$ we have a unified characterization of the boundedness of the commutator in terms of a local oscillatory testing condition on the symbol of the commutator.
The second off-diagonal case $1<q<p<\infty$ was characterised by one of us \cite{HyLpLq} through a global oscillatory condition as
\begin{align}\label{eq:commutatorS}
	\Norm{[b,T]}{L^p\to L^q}\sim \Norm{b}{\dot L^r} := \inf_{c\in\C}\Norm{b-c}{L^r},\quad \frac{1}{q} = \frac{1}{r} + \frac{1}{p}.
\end{align}
The commutator lower bounds of \eqref{eq:commutatorbmo}, \eqref{eq:commutatoralpha} and \eqref{eq:commutatorS} are currently available by two different methods, the first being the approximate weak factorization (awf) argument, the second being the median method. The advantage of both is that they provide a uniform approach to all of the three lower bounds in \eqref{eq:commutatorbmo}, \eqref{eq:commutatoralpha} and \eqref{eq:commutatorS}. As discovered in \cite{HyLpLq}, at their common core lies a minimal notion of non-degeneracy of kernels of singular integrals. Let us next briefly recall the appropriate definitions, thus also fixing the operators of interest in this article.

\begin{defn}\label{defn:SIO:var}
	A singular integral operator (SIO) $T$ is a linear operator $T:\mathcal{S}\to L^1_{\loc}$ on the class of Schwartz functions that has the off-support representation
	\begin{align*}
		Tf(x) = \int_{\R^d} K(x,y)f(y)\ud y,\qquad x\not\in\supp(f),
	\end{align*}
	where the kernel $K:\R^d\times\R^d\setminus  \{x=y\} \to\C$ satisfies the size estimate
	\begin{align}\label{kernel:size}
		\abs{K(x,y)}\leq C\abs{x-y}^{-d},
	\end{align}
	and the smoothness estimates
	\begin{align*}
		\abs{K(x',y)- K(x,y)} + \abs{K(y,x')- K(y,x)}\leq \omega\Big(\frac{\abs{x-x'}}{\abs{x-y}}\Big)\abs{x-y}^{-d}, 
	\end{align*}
whenever $\abs{x-x'}\leq \frac{1}{2}\abs{x-y},$ and 
	where the modulus $\omega$ is increasing and sub-additive, and satisfies  $\omega(0)=0$ and 
	$$\Norm{\omega}{\op{Dini}} = \int_0^1 \omega(t)\frac{\ud t}{t} < \infty.$$
\end{defn}

\begin{defn} A Calderón-Zygmund operator (CZO) is a singular integral operator that is bounded on $L^2(\R^d).$  	We also denote
	\[
	\Norm{T}{\CZO} = \Norm{T}{L^2\to L^2} + \Norm{\omega}{\op{Dini}} + C,
	\]
	where $C$ is the smallest admissible constant in \eqref{kernel:size}.
\end{defn} 

\begin{defn}\label{defn:nondeg} 
A kernel $K(x,y)$ is said to be non-degenerate, if for all $y$ and $r>0$ there exists $x$ such that 
		$$\abs{x-y}\geq r,\qquad \abs{K(x,y)}\gtrsim r^{-d}.$$
%
\end{defn}

\begin{defn}
	An SIO is said to be non-degenerate if its kernel is non-degenerate, and a CZO is said to be non-degenerate if it is a non-degenerate SIO.
\end{defn}

This notion of non-degeneracy was introduced in \cite{HyLpLq} (Definition 2.1.1). It generalizes similar notions in the earlier literature, in particular one due to Stein \cite{MR1232192} (IV.4.6). See \cite{HyLpLq}, Remark 2.1.2 and Examples 2.1.3 to 2.1.6, for a detailed comparison of these notions of non-degeneracy.

It is typical that commutator upper bounds are valid for all Calder\'{o}n-Zygmund operators, while the lower bounds require some non-degeneracy from the singular integral.
We gather everything from the above into the following theorem that fully characterizes, for $p,q\in(1,\infty),$ the boundedness of the commutator between unweighted spaces.
\begin{thm}[\cites{CRW, Janson1978, HyLpLq}]\label{thm:theory1} Let $1<p,q<\infty,$ $b\in L^1_{\loc}(\R^d;\C)$ and $T$ be a non-degenerate Calderón-Zygmund operator. Then, there holds that 
	\begin{align*}
		\Norm{[b,T]}{L^p(\R^d)\to L^q(\R^d)} \sim	
		\begin{cases}
			\Norm{b}{\BMO(\R^d)}, & q = p,\quad (1976)\quad \text{\cite{CRW}}, \\
			\Norm{b}{\dot C^{0,\alpha}(\R^d)},\quad \frac{\alpha}{d} =\frac{1}{p}-\frac{1}{q}, & q>p,\quad (1978) \quad \text{\cite{Janson1978}}, \\
			\Norm{b}{\dot{L}^s(\R^d)},\quad \frac{1}{q} = \frac{1}{s}+\frac{1}{p},  & q<p,\quad  (2021)\quad \text{\cite{HyLpLq}}.
		\end{cases}
	\end{align*}
\end{thm}

In addition to boundedness, we also study the compactness of commutators. 
\begin{defn} A linear operator $T:X\to Y$ between two Banach spaces is said to be compact, provided that for each bounded set $A\subset X,$ the image $TA\subset Y$ is relatively compact, i.e. the closure $\overline{TA}$ is compact in $Y$. We denote by $\calK(X,Y)$ the collection of all compact operators from $X$ to $Y.$
\end{defn}
Recall that all compact operators are bounded.
 Already Uchiyama \cite{Uch1978} in 1978 provided, for a wide class of singular integrals, a characterization of compactness of their commutators in terms of the symbol belonging to the space $\VMO\subset\BMO$ of vanishing mean oscillation (see Definition \ref{defn:VMOw} and set $\alpha = 0, \nu = 1$).
 Without providing the whole background, a special case of Uchiyama's result is the following
 \begin{thm}[\cite{Uch1978}]\label{thm:uchi} Let $b\in L^1_{\loc}$ and $p\in(1,\infty).$ Then, for each $j=1,\dots,d,$ there holds that
 	\begin{align}\label{eq:compact:uch}
 	[b,\calR_j]\in \calK(L^p(\R^d), L^p(\R^d))\quad \mbox{iff} \quad b\in \VMO(\R^d).
 \end{align}
 \end{thm}

The main purpose of this article is to provide two-weight extensions of Theorems \ref{thm:theory1} and \ref{thm:uchi} that fully cover the case $1<p\leq q<\infty.$ 


\subsection{Setup for two-weight bounds}
For integrability parameters $1<p,q<\infty$ and weights $\mu,\lambda:\R^d\to\R_+$ (positive locally integrable functions) our goal is to obtain a characterization of both the boundedness and the compactness of the commutator as a mapping
\begin{align*}
	[b,T]:L^p_{\mu}\to L^q_{\lambda},\qquad \Norm{f}{L^s_{\sigma}} = \Big(\int\abs{f\sigma}^s\Big)^{\frac{1}{s}}.
\end{align*}
This convention of treating the weight as a multiplier goes hand in hand with using the following rescaled weight characteristics
	\begin{align}\label{c:interpret2}
		[\sigma,\omega]_{A_{p,q}} = \sup_{Q\in\calQ}\bave{\sigma^{q}}_Q^{\frac{1}{q}}\bave{\omega^{-p'}}_Q^{\frac{1}{p'}},\qquad [\mu]_{A_{p,q}} = [\mu,\mu]_{A_{p,q}},\qquad [\mu^p]_{A_p} = [\mu]_{A_{p,p}}^p.
	\end{align}
 We say that $\mu\in A_{p,p},$ provided that $[\mu]_{A_{p,p}}<\infty,$ and $\mu\in A_{p},$ provided that $[\mu]_{A_{p}}<\infty.$

 In 1985, under the assumption that $\mu,\lambda\in A_{p,p},$ Bloom \cite{Bloom1985} characterized the two-weight boundedness of the Hilbert commutator on the line by providing the two-sided estimate
\begin{align}\label{eq:Bloom}
 \Norm{b}{\BMO_{\nu}}\lesssim \Norm{[b,H]}{L^p_{\mu}\to L^p_{\lambda}}\lesssim \Norm{b}{\BMO_{\nu}},
\end{align}
where we refer to the left and right estimates as the Bloom lower and upper bounds, respectively, and where \begin{align}\label{eq:weight:bloom}
	\Norm{b}{\BMO_{\nu}}= \sup_{Q\in\calQ}\frac{1}{\nu(Q)}\int_{Q}\abs{b-\ave{b}_Q},\qquad \nu = \mu/\lambda.
\end{align}
Recall that both estimates on the line \eqref{eq:Bloom} depend on the weight characteristics $[\mu]_{A_{p,p}},$ $[\lambda]_{A_{p,p}},$ a fact we will surpress in our notation.
This seminal work of Bloom on the Hilbert transform was extended to general Calder\'on-Zygmund operators by Segovia and Torrea \cite{ST93}.

A renewed interest into such estimates, now commonly called ``of Bloom type'', was sparked by the works of Holmes, Lacey and Wick \cite{Holmes2016,HLW2017} who revisited both results with the recent technology of dyadic representation theorems, due to Petermichl \cite{Petermichl2000} for the Hilbert transform, and due to one of us \cite{Hy1} for general Calder\'on-Zygmund operators.
Only shortly after, Lerner, Ombrosi and Rivera-Ríos \cite{LOR1} obtained a sparse domination of commutators that yielded a simpler proof and also the sharp upper bound, in the sense that the sharp one-weight estimate was reproduced.
At this point, the Bloom lower bound was checked for the vector of the Riesz transform in \cite{HLW2017} (involving all the $d$ Riesz transforms, similarly as above on the line \eqref{eq:commutatorbmo}).
The extension to real-valued homogeneous symbols that do not change sign on some open subset of $\mathbb{S}^{d-1}$ was obtained by Lerner et al. \cite{LOR2}, and the Bloom lower bound for general SIOs of either variable or rough homogeneous kernel, and for complex symbols, was subsequently proved by one of us \cite{HyLpLq}. To complete the picture, recently K.~Li \cite{Li2020multicom} sketched the proof of the Bloom upper bound for rough homogeneous CZOs. All in all, the state-of-the-art qualitative two-weight Bloom boundedness of commutators is recorded as the following

\begin{thm}[\cites{HLW2017, LOR1, HyLpLq, Li2020multicom}]\label{thm:theory2} Let $1<p<\infty$, let $\mu,\lambda\in A_{p,p},$ let $T$ be a non-degenerate Calderón-Zygmund operator and $b\in L^1_{\loc}(\R^d;\C).$ Then, there holds that 
	\begin{align}\label{eq:BloomT}
	 \Norm{[b,T]}{L^p_{\mu}\to L^p_{\lambda}}\sim \Norm{b}{\BMO_{\nu}}.
	\end{align}
\end{thm}

The corresponding characterization of compactness for CZOs of variable kernel was recently obtained by Lacey and J.~Li \cite{LacLi2021}, where they proved the following

\begin{thm}[\cite{LacLi2021}]\label{thm:LacLi2021} Let $1<p<\infty$ and $\mu,\lambda\in A_{p,p},$ let $T$ be a non-degenerate Calderón-Zygmund operator with variable kernel, as in Definition \ref{defn:SIO:var}, and let $b\in L^1_{\loc}(\R^d;\R).$ Then, there holds that 
\begin{equation}\label{eq:LacLi}
	[b,T]\in\calK(L^p_{\mu}, L^p_{\lambda}) \quad \mbox{iff} \quad b\in \VMO_{\nu}.
\end{equation}
\end{thm}

Their proof of the ``if'' part refers to ``the proof of the upper bound for the commutators in \cite{HLW2017}'', which in turn is based on the dyadic representation theorem \cite{Hy1}. In another paper of the same authors with Chen and Vempati \cite{CLLV}, they reprove this result via sparse domination and extend it to spaces of homogeneous type in place of $\R^d$. The ``if'' part is also partially recovered via an abstract method of extrapolation of compactness by Liu, Wu and Yang \cite{LWY}; however, this method leads to assumptions that appear to be stronger than $b\in\VMO_{\nu}$ in dimensions $d\geq 2$, although they can be shown to equivalent in the special dimension $d=1$.

Both \cite{CLLV,LacLi2021} prove the ``only if'' direction of Theorem \ref{thm:LacLi2021} via the median method, which imposes the restriction to real-valued symbols $b$. (See also \cite{LLW}, where the containment of the commutator $[b,H]$ in the still smaller Schatten class $\mathcal{S}_2(L^2_\mu,L^2_\lambda)\subsetneq\mathcal{K}(L^2_\mu,L^2_\lambda)$ is characterized in terms of the membership of the symbol in a suitable Besov space of Bloom type; however, we will not pursue this interesting direction any further in the present work.)

In the fractional setting thus far only one-weight commutator estimates have been studied and only for homogeneous SIOs
\begin{equation*}
  T_\Omega f(x)=\int_{\R^d}\frac{\Omega(x-y)}{\abs{x-y}^d}f(y)\ud y,\qquad\Omega(u)=\Omega(\frac{u}{\abs{u}}),
\end{equation*}
for which Guo, He, Wu and Yang \cite{GuoHeWuYang2021} proved the following

\begin{thm}[\cite{GuoHeWuYang2021}]\label{thm:Guoetal2021} Let $T_{\Omega}$ be a homogeneous $\SIO$ with $\Omega\in L^\infty(\mathbb{S}^{d-1};\C)$ that does not change sign on some open subset of the unit sphere $\mathbb{S}^{d-1}$. Let $b\in L^1_{\loc}(\R^d;\R).$ Let $1<p<q<\infty$ with $\alpha/d = 1/p-1/q$ and suppose that $[w]_{A_{p,q}} < \infty.$ Then, there holds that 
	\begin{align}\label{eq:Guo}
		[b,T_{\Omega}]\in \calK(L^{p}_w, L^q_{w})\quad\mbox{iff}\quad b\in\VMO^{\alpha},
	\end{align}
where $\VMO^{\alpha} = \VMO_{1}^{\alpha},$ see Definition \ref{defn:VMOw} below.
\end{thm}

(The paper \cite{GuoHeWuYang2021} also deals with $\Omega\in L^r(\mathbb{S}^{d-1};\C)$ for finite $r$, but in this case the range of admissible exponent $p,q$ and weights $w$ need to be further restricted; see \cite[Theorem 1.8]{GuoHeWuYang2021} for details.)

The weighted compactness of commutators has also been considered for bilinear operators, see e.g Wang, Xue \cite{WangXue2022}. It would be interesting to know if our methods can be used to extend the results of \cite{WangXue2022} to the multilinear Bloom-type setting, see e.g. \cite{Li2020multicom}.

The recent work of Wen \cite{Wen} also addresses the $L^p_\mu$-to-$L^q_\lambda$ compactness of certain commutators, but the fractional nature of these commutators is in the operator (a fractional integral instead of a CZO) rather than in the symbol; in \cite{Wen}, the characterizing condition for the said compactness is simply $b\in\VMO_\nu$, with the usual Bloom-type $\VMO_\nu$ space. In contrast, we will deal with $L^p_\mu$-to-$L^q_\lambda$ boundedness and compactness of $[b,T]$, where $T$ is a (non-degenerate) CZO, and the characterizing condition will be the containment of $b$ in a suitable fractional $\BMO_\nu^\alpha$ (for boundedness) or $\VMO_\nu^\alpha$ (for compactness). We will give the precise statement in the following section after introducing the necessary notation.

\section{Main definitions and results}

\begin{defn}\label{defn:VMOw}For a weight $\nu$ and a parameter $\alpha\in\R,$ define the space of weighted fractional bounded mean oscillation $\BMO_{\nu}^{\alpha}$  through
	\begin{align}
		\Norm{b}{\BMO_{\nu}^{\alpha}} = \sup_{Q\in\calQ}\calO_{\nu}^{\alpha}(b;Q),\qquad \calO_{\nu}^{\alpha}(b;Q) = \nu(Q)^{-\alpha/d}\Big(\frac{1}{\nu(Q)}\int_Q\abs{b-\ave{b}_Q}\Big).
	\end{align}
%
	Similarly, we define the space $\VMO_{\nu}^{\alpha}\subset \BMO_{\nu}^{\alpha}$ of weighted fractional vanishing mean oscillation through
\begin{align}\label{eq:VMOv1}
	\lim_{r\to 0}\sup_{\ell(Q)\leq r}\calO_{\nu}^{\alpha}(b;Q) = 0
\end{align}
\begin{align}\label{eq:VMOv2}
\lim_{r\to \infty}\sup_{\ell(Q)\geq r}\calO_{\nu}^{\alpha}(b;Q) = 0
\end{align}
\begin{align}\label{eq:VMOv3}
\sup_{Q\in\calQ}\lim_{\abs{x}\to \infty}\calO_{\nu}^{\alpha}(b;Q+x) = 0.
\end{align}
\end{defn}
This particular form of the oscillation $\calO_{\nu}^{\alpha}(b;Q)$ is related to the fact that it is notationally amenable to a weighted fractional John-Nirenberg inequality, see Appendix \ref{sect:appB}.

\begin{rem}\label{rem:vmo3eq} Sometimes \eqref{eq:VMOv3} is replaced with the following a priori stronger uniform condition
\begin{align}\label{eq:VMOv33}
\lim_{r\to \infty}\sup_{Q: \dist(Q,0)> r}\calO_{\nu}^{\alpha}(b;Q) = 0.
\end{align}
In the presence of the conditions \eqref{eq:VMOv1}, \eqref{eq:VMOv2} and under the mild assumption that $\nu$ is doubling, the conditions \eqref{eq:VMOv3} and \eqref{eq:VMOv33} are in fact equivalent. 

Indeed, by the conditions \eqref{eq:VMOv1} and \eqref{eq:VMOv2}, there exists $k>0$ such that $ \ell(Q)\in\R_+\setminus[2^{-k},2^k]$ implies $\calO_{\nu}^{\alpha}(b;Q)\leq \varepsilon.$ Hence it is enough to check \eqref{eq:VMOv33} for cubes with $\ell(Q)\in (2^{-k},2^k),$ see the left-hand side of the estimate \eqref{eq:VMOv3333} below.
We now invoke the method of adjacent dyadic systems, also known as the $\frac13$-trick: there exist $3^d$ dyadic grids $\scrD_i$, where $i=1,\dots,3^d$, such that each cube $Q$ is contained in some $Q'\in\bigcup_{i=1}^{3^d}\scrD_i$ with $\ell(Q')\leq 3\ell(Q)$; see e.g. \cite[Lemma 3.2.26]{HNVW1}. From this and the fact that $\nu$ is doubling, for each $Q$ with $\ell(Q)\in (2^{-k},2^k)$ there exist $\widehat{Q}\in\cup_{j=1}^{3^d}\scrD_i^{k+10}$ such that
\begin{align}\label{eq:x}
	Q\subset \widehat{Q},\qquad \abs{Q}\sim_k\abs{\widehat{Q}},\qquad \calO_{\nu}^{\alpha}(b;Q)\lesssim_{\nu,\alpha,k} \calO_{\nu}^{\alpha}(b;\widehat{Q}).
\end{align}
Let $\{c_j\}$ enumerate the centre points of $\cup_{i=1}^{3^d}\scrD_i^{k+10}$ and denote $Q_k = [-2^{k+9},2^{k+9}]^d.$ Then, by the right-most estimate on line \eqref{eq:x} and \eqref{eq:VMOv3}, we find
\begin{equation}\label{eq:VMOv3333}
\begin{split}
  \lim_{r\to\infty}\sup_{\substack{Q: \dist(Q,0)> r \\ \ell(Q)\in (2^{-k},2^k)}}\calO_{\nu}^{\alpha}(b;Q) 
  &\lesssim_{\nu,\alpha,k} \lim_{r\to\infty}\sup_{\substack{Q: \dist(Q,0)> r \\ \ell(Q)\in (2^{-k},2^k)}}\calO_{\nu}^{\alpha}(b;\widehat{Q})    \\
	&	=\lim_{j\to \infty}\calO_{\nu}^{\alpha}(b;Q_k+c_j) = 0,
\end{split}
\end{equation}
and this concludes the proof that \eqref{eq:VMOv3} implies \eqref{eq:VMOv33}.
\end{rem}

%
%
%
%
%
%
%
%
%

\begin{defn}\label{defn:bloom:weight} Given two weights $\mu,\lambda$ and exponents $1<p,q<\infty,$ we define the Bloom weight 
	\begin{align}
		\nu =\nu_{p,q} = \big(\mu/\lambda\big)^{\frac{1}{1/p+1/q'}} =  \big(\mu/\lambda\big)^{\frac{1}{1+\alpha/d}}\qquad \alpha/d= 1/p-1/q.
	\end{align}
\end{defn}

Our main result is the following.
\begin{thm}\label{thm:main1} Let $T$ be a non-degenerate Calderón-Zygmund operator 
	let $b\in L^1_{\loc}(\R^d;\C).$ Let $1<p\leq q<\infty,$ let $\alpha/d = 1/p-1/q,$ let $\mu\in A_{p,p}$ and $\lambda\in A_{q,q}$ and let $\nu = \nu_{p,q}$ be the fractional Bloom weight as in Definition \ref{defn:bloom:weight}. 	
Then, there holds that 
	\begin{align}\label{eq:main1}
			\Norm{[b,T]}{L^p_{\mu}\to L^q_{\lambda}}\sim \Norm{b}{\BMO_{\nu}^{\alpha}}
	\end{align}
%
 and
	\begin{align}\label{eq:main2}
			[b,T]\in \calK(L^p_{\mu}, L^p_{\lambda})\quad\mbox{iff}\quad b\in\VMO_{\nu}^{\alpha}.
	\end{align}
\end{thm}

\subsection{Interpretation in the one weight setting} Since we are assuming $1<p\leq q,$ it especially follows directly from Hölder's inequality that $[w]_{A_{q,q}}, [w]_{A_{p,p}}\leq [w]_{A_{p,q}},$ hence $A_{p,q}\subset A_{p,p}\cap A_{q,q},$ and we extend Theorem \ref{thm:Guoetal2021} to Calderón-Zygmund operators with variable kernel.

\begin{cor}\label{cor:oneweight} Let $T$ be a non-degenerate Calderón-Zygmund operator and $b\in L^1_{\loc}(\R^d;\C).$ Let $1<p\leq q<\infty$ with $\alpha/d = 1/p-1/q$ and suppose that $w\in A_{p,p}\cap A_{q,q}.$ Then, there holds that 
	\begin{align*}
		\Norm{[b,T]}{L^p_{w}\to L^q_{w}}\sim \Norm{b}{\BMO^{\alpha}}
	\end{align*}
and
	\begin{align*}
		[b,T]\in \calK(L^{p}_w, L^q_{w})\quad\mbox{iff}\quad b\in\VMO^{\alpha}.
	\end{align*}
\end{cor}

%

\subsection{Notation}\label{sect:notation}
\begin{itemize}
%
	\item Whenever we have fixed exponents $p,q$ and weights $\mu,\lambda,$ we will always without exception denote $\alpha/d = 1/p-1/q$ and $\nu = \nu_{p,q},$ as in Definition \ref{defn:bloom:weight}.
	
	\item When $p=q$ and $\alpha = 0$ we drop the superscript $\alpha.$
	
	\item When $w = 1$ is the Lebesgue measure, we do not mark it.
	\item We denote $L^1_{\loc}(\R^d;\C) = L^1_{\loc},$ $\int_{\R^d} = \int,$ and so on, mostly leaving out the ambient space whenever this information is obvious or unimportant.
	\item We denote averages with
	$
	\langle f \rangle_A = \fint_A f= \frac{1}{|A|} \int_A f,
	$
	where $|A|$ denotes the Lebesgue measure of the set $A$. The indicator function of a set $A$ is denoted by $1_A$. 
	
	\item We denote various dyadic grids with the symbol $\mathscr{D}.$
	\item We denote various sparse collections with the symbol $\mathscr{S}.$
	\item We denote $A \lesssim B$, if $A \leq C B$ for some constant $C>0$ depending only on the dimension of the underlying space, on integration exponents, on sparse constants, on constants from the kernel estimates, on constants depending on the weights $\mu,\lambda,$ and on other absolute constants appearing in the assumptions that we do not care about.
	 Then  $A \sim B$, if $A \lesssim B$ and $B \lesssim  A.$ Moreover, subscripts on constants ($C_{a,b,c,...}$) and quantifiers ($\lesssim_{a,b,c,...}$) signify their dependence on those subscripts. 	
	 \item 	
	 We again emphasize that the implicit constants in the previous point are allowed to depend on the weights $\mu$ and $\lambda,$ hence our theorem statements appear e.g. as Theorem \ref{thm:main1} and Corollary \ref{cor:oneweight} do. Sometimes in order to make the proofs easier to follow we indicate the weight dependence in the implicit constant as $\lesssim_{[\mu]_{A_{p,p}},[\lambda]_{A_{q,q}}},$ for example.
\end{itemize}

\section{On the Bloom weight}
In the following Proposition \ref{prop:bloom1} we gather the basic properties of the fractional Bloom weight of Definition \ref{defn:bloom:weight}.

\begin{prop}\label{prop:bloom1} Suppose that $1<p,q<\infty,$ $\mu\in A_{p,p}$ and $\lambda\in A_{q,q}.$ Then, there holds that 
	\begin{align}\label{eq:bloom1}
		1\leq \frac{\mu^p(Q)^{\frac{1}{p}}\lambda^{-q'}(Q)^{\frac{1}{q'}}}{\nu(Q)^{\frac{1}{p}+\frac{1}{q'}}}\leq [\mu]_{A_{p,p}} [\lambda]_{A_{q,q}}
	\end{align}
	and 
	\begin{align}\label{eq:bloom111}
		[\nu]_{A_{s(p,q)}}^{2/s(p,q)} \leq  [\mu]_{A_{p,p}}[\lambda]_{A_{q,q}},\qquad \nu^{1/s(p,q)}=(\mu/\lambda)^{\frac{1}{2}},
	\end{align} 
	where
	\begin{align}\label{eq:bloom1111}
		s(p,q) = 1 + \frac{1/q+1/p'}{1/p+1/q'} = \frac{2}{1+\alpha/d}.
	\end{align}
\end{prop}

\begin{rem} There holds that 
	\begin{align*}
		s(p,q) < 2,\quad if\quad p<q;\qquad
		s(p,q) = 2,\quad if\quad p=q; \qquad
		s(p,q) > 2,\quad if\quad p>q,
	\end{align*}
and especially for each $1<p,q<\infty$ that $\nu = \nu_{p,q}\in A_{\infty}.$ 
\end{rem}
\begin{proof}  The first identity on line \eqref{eq:bloom1111} will be automatically checked while verifying the left claim on line \eqref{eq:bloom111}, while the second is simple algebra. Next, we check \eqref{eq:bloom1}.
	By Hölders inequality applied with the exponents
	\[
	\frac{1/p}{1/p+1/q'}+\frac{1/q'}{1/p+1/q'} = 1
	\] we have 
	\begin{align*}
		\ave{\nu}_Q^{1/p+1/q'} &= \Bave{\big(\frac{\mu}{\lambda}\big)^{(1/p+1/q')^{-1}}}_Q^{1/p+1/q'} \\ 
		&\leq \Big(\bave{\mu^{p}}_Q^{\frac{1/p}{1/p+1/q'}}\bave{\lambda^{-q'}}_Q^{\frac{1/q'}{1/p+1/q'}}\Big)^{1/p+1/q'} = \ave{\mu^p}_Q^{\frac{1}{p}}\ave{\lambda^{-q'}}_Q^{\frac{1}{q'}}.
	\end{align*}	
	For the other direction, suppose that $\mu\in A_{p,p},$ $\lambda\in A_{q,q}.$  Then, there holds that 
	\begin{align*}
		\ave{\mu^p}_Q^{\frac{1}{p}}\ave{\lambda^{-q'}}_Q^{\frac{1}{q'}} \leq [\mu]_{A_{p,p}} [\lambda]_{A_{q,q}}
		\big( 	\ave{\mu^{-p'}}_Q^{\frac{1}{p'}}\ave{\lambda^{q}}_Q^{\frac{1}{q}} \big)^{-1},
	\end{align*}
	and by Hölder's inequality
	\begin{align*}
		\Big( 	\ave{\mu^{-p'}}_Q^{\frac{1}{p'}}\ave{\lambda^{q}}_Q^{\frac{1}{q}} \Big)^{-1} \leq \Bave{\big(\frac{\lambda}{\mu}\big)^{\frac{1}{1/q+1/p'}}}_Q^{-(1/q+ 1/p')} = \Big( \bave{\nu^{-t(p,q)}}_Q^{-\frac{1}{t(p,q)}} \Big)^{1/p+1/q'},
	\end{align*} 
	where we denote $t(p,q) = \frac{1/p+1/q'}{1/q+1/p'}.$
	Since $x\mapsto x^{-\beta}$ is convex, for any $\beta\geq 0,$ by Jensen's inequality
	\begin{align*}
		\big( \bave{\nu^{-\frac{1}{t(p,q)}}}_Q^{-t(p,q)} \big)^{\frac{1}{p}+\frac{1}{q'}}\leq  \ave{\nu}_Q^{\frac{1}{p}+\frac{1}{q'}}
	\end{align*}
	and we conclude that 
	\begin{align*}
		\ave{\mu^p}_Q^{\frac{1}{p}}\ave{\lambda^{-q'}}_Q^{\frac{1}{q'}} \leq [\mu]_{A_{p,p}} [\lambda]_{A_{q,q}}\ave{\nu}_Q^{\frac{1}{p}+\frac{1}{q'}}.
	\end{align*}
	Hence, we have checked both sides of \eqref{eq:bloom1}.
	
	We already saw above that 
	\begin{align*}
		\ave{\nu}_Q^{1/p+1/q'} \leq \ave{\mu^p}_Q^{\frac{1}{p}}\ave{\lambda^{-q'}}_Q^{\frac{1}{q'}},\qquad 	\bave{\nu^{-t(p,q)}}_Q^{\frac{1/p+1/q'}{t(p,q)}}\leq 	\ave{\mu^{-p'}}_Q^{\frac{1}{p'}}\ave{\lambda^{q}}_Q^{\frac{1}{q}}
	\end{align*}
	and multiplying these estimates together, then raising to the power $(1/p+1/q')^{-1},$ gives
	\begin{align}\label{eq:nu3}
		\bave{\nu}_Q 	\bave{\nu^{-t(p,q)}}_Q^{\frac{1}{t(p,q)}} \leq  \big([\mu]_{A_{p,p}}[\lambda]_{A_{q,q}}\big)^{(1/p+1/q')^{-1}}< \infty. 
	\end{align}
	Since $(s\mapsto s'/s):(1,\infty)\to (0,\infty)$ is bijective, \eqref{eq:nu3} shows that $\nu\in A_{s(p,q)}$ for the unique $s(p,q)\in (1,\infty)$ such that $\frac{s(p,q)'}{s(p,q)} := t(p,q).$ Solving for $s(p,q),$ we find
	\begin{align}
		s(p,q) = 1 + t(p,q)^{-1} = 1 + \frac{1/q+1/p'}{1/p+1/q'},
	\end{align}
	and by \eqref{eq:nu3} that $[\nu]_{A_{s(p,q)}} \leq ([\mu]_{A_{p,p}}[\lambda]_{A_{q,q}})^{1/(1/p+1/q')}.$
	By $s(p,q)(1/p+1/q') = 2,$ we have
	\begin{align}
		(\mu/\lambda)^{\frac{1}{2}} = (\mu/\lambda)^{(1/p+1/q')^{-1}s(p,q)^{-1}} = \nu^{1/s(p,q)}.
	\end{align}
	Denoting $s = s(p,q)$ and using that $[\nu]_{A_{s}} = [\nu^{1/s}]_{A_{s,s}}^{s}$, we find that 
	\begin{align*}
		[(\mu/\lambda)^{\frac{1}{2}}]_{A_{s,s}} = [\nu^{\frac{1}{s}}]_{A_{s,s}} = [\nu]_{A_s}^{\frac{1}{s}}\leq \big([\mu]_{A_{p,p}}[\lambda]_{A_{q,q}}\big)^{(1/p+1/q')^{-1}s^{-1}} = \big([\mu]_{A_{p,p}}[\lambda]_{A_{q,q}}\big)^{\frac{1}{2}},
	\end{align*} 
	which gives \eqref{eq:bloom111}.
\end{proof}

\begin{defn} Let $\mu,\lambda$ be weights and $b\in L^1_{\loc}(\R^d;\C)$ and define 
	\begin{align}\label{norm:2wBMOsub}
		\Norm{b}{\BMO^{p,q}_{\mu,\lambda}} = \sup_{Q\in\calQ}\frac{1}{\mu^p(Q)^{\frac{1}{p}}\lambda^{-q'}(Q)^{\frac{1}{q'}}}\int_Q\abs{b-\ave{b}_Q}
	\end{align}
\end{defn}
As an immediate corollary of line \eqref{eq:bloom1} we get the following
\begin{prop}\label{prop:bmoeq} Suppose that $1<p,q<\infty$ and $\mu\in A_{p,p},$ $\lambda\in A_{q,q}$ and let $\alpha/d = 1/p-1/q.$ Let $b\in L^1_{\loc}(\R^d;\C).$ Then, there holds that
\begin{align}
	\Norm{b}{\BMO_{\nu}^{\alpha}}\sim	\Norm{b}{\BMO^{p,q}_{\mu,\lambda}}.
\end{align}
\end{prop}

\section{Boundedness} 

\subsection{Lower bound}
Different versions of the following Proposition \ref{prop:osc1} appear throughout the literature, the one with complex valued symbols attained through the approximate weak factorization argument is from \cite{HyLpLq}. Before formulating it, let us fix the following convention. If a function $\psi$ has support on $Q,$ then we write $\psi = \psi_Q$, in particular if $\psi$ is supported on a major subset $E = E_Q\subset Q,$ we write $\psi = \psi_{E} = \psi_{E_Q}.$

\begin{defn} A collection of sets $\mathscr{S}$ is said to be $\gamma$-sparse, for $\gamma\in(0,1),$ if there exists a pairwise disjoint collection $\mathscr{S}_E = \big\{U_E : U_E\subset U \in\mathscr{S},\,\abs{U_E}>\gamma \abs{U}\big\}.$ 
Furthermore, when $\abs{U_E}>\gamma \abs{U},$  we speak of $\gamma$-major subsets. When the parameter $\gamma$ is of no consequence, we speak of major subsets and sparse collections.
\end{defn}

\begin{prop}[\cite{HyLpLq}]\label{prop:osc1} Let $T$ be a non-degenerate singular integral and $b\in L^1_{\loc}(\R^d;\C).$  Let $Q$ be a fixed cube. Then, there exists a cube $\wt{Q}$ such that  $\dist(Q,\wt{Q})\sim \ell(Q) = \ell(\wt{Q})$ and for any $\gamma$-major subsets $E\subset Q$ and $\wt{E}\subset \wt{Q}$ we have
	\begin{align}\label{osc1}
		\int_{E}\abs{b-\ave{b}_{E}} \lesssim \babs{\bave{[b,T]g_{\wt{E}},h_{E}}} + \babs{\bave{[b,T]h_{\wt{E}},g_E}},
	\end{align}
where the auxiliary functions satisfy
	\begin{align*}
		g_{E}= 1_{E},\qquad g_{\wt{E}} = 1_{\wt{E}},\qquad 	\abs{h_{E}}\lesssim 1_{E},\qquad \abs{h_{\wt{E}}}\lesssim 1_{\wt{E}},
	\end{align*}
	and all the above implicit constants depend only on the kernel of $T$ and the parameter $\gamma.$
\end{prop}

\begin{prop}\label{prop:bdd:lb}Let $1<p,q<\infty$, let $\mu\in A_{p,p},\lambda\in A_{q,q},$ let $\alpha/d = 1/p-1/q,$ let $T$ be a non-degenerate singular integral and $b\in L^1_{\loc}(\R^d;\C).$  Then, there holds that 
	\begin{align}
		\Norm{b}{\BMO_{\nu}^{\alpha}}	\lesssim \Norm{[b,T]}{L^p_{\mu}\to L^q_{\lambda}},
	\end{align}
	where the implicit constant depends only on the weights and the kernel of $T.$
\end{prop}
\begin{proof} By Proposition \ref{prop:bmoeq} it is enough to check the lower bound with $\Norm{b}{\BMO^{p,q}_{\mu,\lambda}}$ in place of $\Norm{b}{\BMO^{\alpha}_{\nu}}.$ In Proposition \ref{prop:osc1} we let $E = Q$ and $\wt{E} = \wt{Q}$ and then estimate the right-hand side of \eqref{osc1} as
	\begin{align*}
		\babs{\bave{[b,T]g_{\wt{Q}},h_{Q}}} &= \babs{\bave{\lambda[b,T]g_{\wt{Q}},h_{Q}\lambda^{-1}}} 
		\lesssim \bNorm{[b,T]g_{\wt{Q}}}{L^q_{\lambda}}\bNorm{h_Q}{L^{q'}_{\lambda^{-1}}} \\
		&\leq \bNorm{[b,T]}{L^p_{\mu}\to L^q_{\lambda}}\bNorm{g_{\wt{Q}}}{L^p_{\mu}}\bNorm{h_Q}{L^{q'}_{\lambda^{-1}}}  \lesssim \bNorm{[b,T]}{L^p_{\mu}\to L^q_{\lambda}} \big( \mu^p(\wt{Q})\big)^{\frac{1}{p}}\big(\lambda^{-q'}(Q)\big)^{\frac{1}{q'}}.
	\end{align*}
	The second term is estimated identically.
	To conclude it remains to show that $$\mu^p(\wt{Q})\lesssim_{ [\mu]_{A_{p,p}}}\mu^p(Q),$$ which follows easily from the definition of $A_{p,p}$ weights, or from the doubling property.
\end{proof}

\subsection{Upper bound}
We turn to the upper bound. 
We will involve the unweighted sparse operator and the two-weight fractional sparse operator
\begin{align*}
	\mathcal{A}_{\mathscr{S}}f =  	\mathcal{A}(f;\mathscr{S}) =\sum_{Q\in\mathscr{S}}\ave{f}_Q1_Q,\qquad 	\mathcal{A}_{\mu,\lambda}^{p,q}(f;\mathscr{S}) = \sum_{P\in\mathscr{S}}\frac{\mu^p(P)^{\frac{1}{p}}\lambda^{-q'}(P)^{\frac{1}{q'}}}{\abs{P}}\ave{f}_P1_P.
\end{align*}

%

The following sparse domination of commutators was first obtained by Lerner, Ombrosi, Rivera-Ríos \cite{LOR1} as a step towards a Bloom type upper bound for commutators.
\begin{lem}\label{lem:sdom} Let $T$ be a Calderón-Zygmund operator. Suppose that $f\in L^1_{\loc}.$ Then, there exist sparse collections $\mathscr{S}_j\subset\mathscr{D}_j$, where $j=1,\dots,3^d$, such that 
	\begin{align}
		\abs{[b,T]f} \lesssim \sum_{j=1}^{3^d}\Big(\mathcal{A}_{b,\mathscr{S}_j}\abs{f} + \mathcal{A}^*_{b,\mathscr{S}_j}\abs{f}\Big),\qquad  \mathcal{A}^*_{b,\mathscr{S}_j}f=  \sum_{Q\in\script{S}_j} \bave{\abs{b-\ave{b}_Q}f}_Q1_Q.
	\end{align}
	where $\mathcal{A}_{b,\mathscr{S}_j}$ is the adjoint of $\mathcal{A}_{b,\mathscr{S}_j}^*.$ Sometimes we denote $\mathcal{A}^*_{b,\mathscr{S}_j}f = \mathcal{A}^*_{b}(f;\mathscr{S}_j).$
\end{lem}
 The sparse collections in Lemma \ref{lem:sdom} depend on the initial function $f,$ but the implicit constants are independent of $f.$
Furthermore, Lerner et al. \cite{LOR1} implicitly provided the following augmentation.
\begin{lem}\label{lem:aug} Let $\mathscr{S}\subset\scrD$ be an arbitrary sparse collection of cubes. Then, for each function $b\in L^1_{\loc},$ there exists a sparse collection sandwiched as $\scrD\supset \widehat{\mathscr{S}} \supset \mathscr{S}$ such that for each $Q\in\mathscr{S}$
	\begin{align*}
	\abs{b-\ave{b}_{Q}}1_{Q}\lesssim \sum_{\substack{P\in \widehat{\mathscr{S}} \\ P\subset Q}}\bave{\abs{b-\ave{b}_P}}_P1_P.
\end{align*}
\end{lem}

\begin{thm}\label{thm:ub} Let $T$ be a Calderón-Zygmund operator, let $1<p\leq q<\infty$ and $\mu\in A_{p,p}$ and $\lambda\in A_{q,q}.$ Then, there holds that 
	\[
\Norm{[b,T]}{L^p_{\mu}\to L^q_{\lambda}} \lesssim	\Norm{b}{\BMO_{\nu}^{\alpha}}.
	\]
\end{thm}

\begin{proof}  By Lemma \ref{lem:sdom} 
	\begin{align*}
		\Norm{[b,T]}{L^p_{\mu}\to L^q_{\lambda}} &\lesssim 	\Norm{\mathcal{A}_{b,\mathscr{S}}^*}{L^p_{\mu}\to L^q_{\lambda}} + \Norm{\mathcal{A}_{b,\mathscr{S}}}{L^p_{\mu}\to L^q_{\lambda}}
	\end{align*}
and by  $ \Norm{\mathcal{A}_{b,\mathscr{S}}}{L^p_{\mu}\to L^q_{\lambda}} = \Norm{\mathcal{A}_{b,\mathscr{S}}^*}{L^{q'}_{\lambda^{-1}}\to L^{p'}_{\mu^{-1}}},$ it is enough to show the upper bound for $\mathcal{A}_{b,\mathscr{S}}^*.$
By Lemma \ref{lem:aug} and Proposition \ref{prop:bmoeq} we have 
	\begin{equation}\label{eq:sub:sharp1}
		\begin{split}
			\bave{\abs{b-\ave{b}_Q}\abs{f}}_Q&\lesssim \Bave{\sum_{\substack{P\in\widehat{\mathscr{S}}\\ P\subset Q}}\bave{\abs{b-\ave{b}_P}}_P\abs{f}1_P}_Q = \Bave{\sum_{\substack{P\in\widehat{\mathscr{S}}\\ P\subset Q}}\bave{\abs{b-\ave{b}_P}}_P\ave{\abs{f}}_P1_P}_Q\\
			& \lesssim_{[\mu]_{A_{p,p}},[\lambda]_{A_{q,q}}}  	\Norm{b}{\BMO_{\nu}^{\alpha}}\Bave{\sum_{\substack{P\in\widehat{\mathscr{S}}\\ P\subset Q}}\frac{\mu^p(P)^{\frac{1}{p}}\lambda^{-q'}(P)^{\frac{1}{q'}}}{\abs{P}}\ave{\abs{f}}_P1_P}_Q
		\end{split}
	\end{equation}	
	and hence by the boundedness of the sparse operator (see e.g. \cite{Lerner2015}) we find
	\begin{align*}
		\Norm{\mathcal{A}_{b,\mathscr{S}}^*}{L^p_{\mu}\to L^q_{\lambda}}  \lesssim_{[\mu]_{A_{p,p}},[\lambda]_{A_{q,q}}} 	\Norm{b}{\BMO_{\nu}^{\alpha}}\bNorm{f\mapsto\mathcal{A}_{\mu,\lambda
			}^{p,q}(f;\widehat{\mathscr{S}})}{L^p_{\mu}\to L^q_{\lambda}}.
	\end{align*}
	Now what remains is to estimate as follows
	\begin{align*}
		\babs{	\bave{	\mathcal{A}_{\mu,\lambda}^{p,q}(f;\widehat{\mathscr{S}}) , g}} &\leq \sum_{Q\in\widehat{\mathscr{S}}} \mu^p(Q)^{\frac{1}{p}}\ave{\abs{f}}_Q \lambda^{-q'}(Q)^{\frac{1}{q'}}\ave{\abs{g}}_Q \\
		&\leq \big( \sum_{Q\in\widehat{\mathscr{S}}} \ave{\abs{f}}_Q^q\mu^p(Q)^{\frac{q}{p}}\big)^{\frac{1}{q}}\big(\sum_{Q\in\widehat{\mathscr{S}}}\ave{\abs{g}}_Q^{q'}\lambda^{-q'}(Q)\big)^{\frac{1}{q'}} \\
		&\leq  \big( \sum_{Q\in\widehat{\mathscr{S}}} \ave{\abs{f}}_Q^p\mu^p(Q)\big)^{\frac{1}{p}}\big(\sum_{Q\in\widehat{\mathscr{S}}}\ave{\abs{g}}_Q^{q'}\lambda^{-q'}(Q)\big)^{\frac{1}{q'}} \\
		&\lesssim [\mu]_{A_{p,p}}^{p'}[\lambda^{-1}]_{A_{q',q'}}^{q}\Norm{f}{L^p_{\mu}}\Norm{g}{L^{q'}_{\lambda^{-1}}} =  [\mu]_{A_{p,p}}^{p'}[\lambda]_{A_{q,q}}^q\Norm{f}{L^p_{\mu}}\Norm{g}{L^{q'}_{\lambda^{-1}}},
	\end{align*}
	where we used $\Norm{\cdot}{\ell^q}\leq\Norm{\cdot}{\ell^p}$ (by $p\leq q$) and Lemma \ref{lem:prePyth}.
\end{proof}

\section{Compactness}
\subsection{Sufficiency}
	Let $T$ be a Calderón-Zygmund operator. In this section we show that  $b\in\VMO_{\nu}^{\alpha}$ implies $[b,T]\in\calK(L^p_{\mu},L^q_{\lambda}).$ We mimic the proof from Lacey and Li in \cite{LacLi2021} and we provide all details, some of which are different. Formally the idea is to show that
	\begin{align}\label{eq:goal}
			T = T_c + T_{\varepsilon},\qquad [b,T_c]\in \calK(L^p_{\mu},L^q_{\lambda}),\qquad \lim_{\varepsilon\to 0}\Norm{	[b,T_{\varepsilon}]}{L^p_{\mu}\to L^q_{\lambda}} = 0.
	\end{align}
	Then, we would be done by the fact that compact linear operators form a closed subspace of all bounded linear operators.
	
	
We select a bump such that
	\begin{align*}
	\varphi\in C^{\infty}_c(\R^d;[0,1]),\qquad 	\varphi(x) = \begin{cases}
			1,\quad \abs{x} \leq\frac{1}{2}, \\ 
			0,\quad \abs{x}\geq 1.
		\end{cases}
	\end{align*}
We define
\begin{align*} 
	\varphi_x^{0,r}(y) = \varphi\big(\frac{x-y}{r}\big)
\end{align*}
and with $0<r<R<\infty$ partition unity as 
\begin{align*}
	1 = \varphi_x^{0,r} +  \varphi_x^{r,R} +  \varphi_x^{R,\infty},\qquad  \varphi_x^{r,R} =  \varphi_x^{0,R}-  \varphi_x^{0,r},\qquad \varphi_x^{R,\infty} = 1 - \varphi_x^{0,R}.
\end{align*}
We also denote 
\[
	\varphi^{a,b} = \varphi_0^{a,b},\qquad a,b\in[0,\infty].
\]
Then, we decompose
\begin{align*}
	T &= \big(\varphi^{0,R}+\varphi^{R,\infty}\big)T \big(\varphi^{0,R}+\varphi^{R,\infty}\big) \\ 
	&= \varphi^{0,R}T \varphi^{0,R} 
	+ \Big[  \varphi^{R,\infty}T \varphi^{R,\infty}
	+ \varphi^{R,\infty}T \varphi^{0,R}
	+ \varphi^{0,R}T \varphi^{R,\infty} \Big] =  T^{0,R} + T^{R,\infty}.
\end{align*}
The term $T^{R,\infty}$ is good as it is and the other we decompose further:
\begin{align*}
	T^{0,R} = 	T^{0,R}\big( \varphi_x^{0,r} + 	\varphi_x^{r,10R} +	\varphi_x^{10R,\infty} \big) = 	T^{0,R}\varphi_x^{0,r} + 	T^{0,R}\varphi_x^{r,10R},
\end{align*}
where we noted that $T^{0,R}\varphi_x^{10R,\infty}  = 0.$ In total, our decomposition is the following
\begin{align}
 	T =  T_c + T_{\varepsilon},\qquad  T_c = 	T^{0,R}\varphi_x^{r,10R},\qquad T_{\varepsilon} = 	T^{0,R}\varphi_x^{0,r} + T^{R,\infty},
\end{align}
where we now take the convention of writing $\varepsilon = r = R^{-1}.$

The rest of this section is devoted to proving the two claims on line \eqref{eq:goal}, and we begin with
\begin{prop} Let $1<p\leq q<\infty$, $\mu\in A_{p,p}$ and $\lambda\in A_{q,q}$. Suppose that $b\in\BMO_{\nu}^{\alpha}.$ Then
	\[
	[b,	T_c] \in \calK(L^p_{\mu},L^q_{\lambda}).
	\]
\end{prop}
\begin{proof}
%
	Let us denote $Q=[-R,R]^d,$ $S=10R$ and 
	\[
	T_Lf(x)=\int L(x,y)f(y)\ud y,\qquad L(x,y) = K(x,y)\varphi_x^{r,S}(y).
	\]
%
so that 
	\begin{align*}
 	T_c = 	T^{0,R}\varphi_x^{r,S} = \varphi^{0,R}\circ T_L\circ \varphi^{0,R}.
	\end{align*}
We view the commutator as
	\begin{align*}
		[b,T_c] &=  \varphi^{0,R}\circ[b, T_L] \circ \varphi^{0,R} \\ 
	&=  \varphi^{0,R}\circ[b-\ave{b}_Q, T_L] \circ \varphi^{0,R} \\
		&= \big((b-\ave{b}_Q)\varphi^{0,R}\circ T_L \circ \varphi^{0,R} \big)- \big(\varphi^{0,R}\circ  T_L \circ (b-\ave{b}_Q\big)\varphi^{0,R})
	\end{align*}
and show that both terms separately are compact. More precisely, in the diagrams
	\begin{equation}\label{eq:diag1}
		L^p_\mu\overset{\varphi^{0,R}(b-\ave{b}_{Q})}{\longrightarrow} L^1(Q)\overset{T_L}{\longrightarrow}C\big(Q;\Norm{\cdot}{\infty}\big)\overset{\varphi^{0,R} }{\longrightarrow} L^q_\lambda
	\end{equation}
and	
	\begin{equation}\label{eq:diag2}
		L^p_\mu\overset{\varphi^{0,R}}{\longrightarrow} L^1(Q)\overset{T_L}{\longrightarrow} C\big(Q;\Norm{\cdot}{\infty}\big)\overset{(b-\ave{b}_{Q})\varphi^{0,R} }{\longrightarrow} L^q_\lambda
	\end{equation}
we show that the first and the last maps are bounded and that the one in the middle is compact.  We only provide details for the diagram \eqref{eq:diag1} with \eqref{eq:diag2} being completely analogous. We first show the compactness of $T_L.$ Let there be a family $\{f_j\}_j$ such that $\sup_j\Norm{f_j}{L^1(Q)}\lesssim 1;$ then we need to show that the collection $\{T_Lf_j\}_j\subset C\big(Q;\Norm{\cdot}{\infty}\big)$ has a converging subsequence.
By the Arzela-Ascoli theorem, it is enough to check equicontinuity and equiboundedness of $\{T_Lf_j\}_j\subset C\big(Q;\Norm{\cdot}{\infty}\big).$ Equiboundedness is immediate by the simple estimate 
\begin{align*}
		\babs{T_Lf(x)} \leq \int_{Q} \abs{K(x,y)\varphi_x^{r,S}(y)}\abs{f(y)}\ud y\lesssim r^{-d}\Norm{f}{L^1(Q)}.
\end{align*}

For equicontinuity, we first check that $L(x,y) = \varphi_x^{r,S}(y)K(x,y)$ is a CZ-kernel, which can be seen as follows. Recall that $L(x,y) = \big(\varphi_x^{0,S}-\varphi_x^{0,r}\big)K(x,y)$ and combine this with the fact that $\varphi_x^{0,t}(y)K(x,y)$ is a CZ-kernel, for $t>0,$ for this see the estimate  \eqref{ghj} below, and that CZ-kernels are closed under summation.
Let  $x,x'\in Q$ be such that $\abs{x-x'} <  sr,$ for $0<s<\frac{1}{4}.$ Then, we have 
\begin{align*}
	\abs{T_Lf(x)-T_Lf(x')} \leq \sup_{y\in Q}\abs{L(x,y)-L(x',y)}\Norm{f}{L^1(Q)}
\end{align*}
and using the regularity estimate of $L$ for the factor in front,
\begin{align*}
	\abs{L(x,y)-L(x',y)} &= \abs{L(x,y)-L(x',y)}1_{B(x,r/2)^c}(y) \\ 
	&\lesssim \omega\big( \frac{\abs{x-x'}}{\abs{x-y}}\big)\abs{x-y}^{-d}1_{B(x,r/2)^c}(y) \lesssim \omega(s)r^{-d}.
\end{align*}
This shows equicontinuity.

	The rightmost map of diagram \eqref{eq:diag1} is bounded:
	\[
	\|\varphi^{0,R}f\|_{L_\lambda^q}\leq \big(\int_{Q} |f|^q\lambda^q\big)^{1/q}\leq \|f\|_{L^\infty} \big(\int_{Q} \lambda^q\big)^{1/q} \lesssim  \|f\|_{L^\infty}.
	\]
	For the first map, we have 
	\begin{align*}
		\Norm{\varphi^{0,R}(b-\ave{b}_{Q})f}{L^1(Q)} \leq \Norm{b-\ave{b}_Q}{L^{p'}_{\mu^{-1}}(Q)} \Norm{f}{L^{p}_{\mu}}
	\end{align*}
	and it remains to show the finiteness of the term in front. Using Lemma \ref{lem:aug} with a single cube $Q,$ then Lemma \ref{lem:Pyth} and the fact that $b\in\BMO_{\nu}^{\alpha} = \BMO_{\mu,\lambda}^{p,q},$ we find
	\begin{align*}
		\Norm{b-\ave{b}_Q}{L^{p'}_{\mu^{-1}}(Q)}  &\lesssim \Norm{\sum_{P\in\scrS(Q)}\ave{\abs{b-\ave{b}_P}}1_P}{L^{p'}_{\mu^{-1}}(Q)} \\  
		&\leq \Norm{b}{\BMO_{\mu,\lambda}^{p,q}}\Big(\sum_{P\in\scrS(Q)}\BNorm{\frac{\mu^p(P)^{\frac{1}{p}}\lambda^{-q'}(P)^{\frac{1}{q'}}}{\abs{P}}1_P}{L^{p'}_{\mu^{-1}}}^{p'} \Big)^{\frac{1}{p'}}.
	\end{align*}
	It remains to show that the inside sum is finite. For this, we have 
	\begin{align*}
		\sum_{P\in\scrS(Q)}\BNorm{\frac{\mu^p(P)^{\frac{1}{p}}\lambda^{-q'}(P)^{\frac{1}{q'}}}{\abs{P}}1_P}{L^{p'}_{\mu^{-1}}}^{p'} &\leq [\mu]_{A_{p,p}}^{p'}	\sum_{P\in\scrS(Q)}	\lambda^{-q'}(P)^{\frac{p'}{q'}} \\ 
		&\lesssim \Big(\sum_{P\in\scrS(Q)}	\lambda^{-q'}(P)\Big)^{\frac{p'}{q'}} \lesssim \lambda^{-q'}(Q)^{\frac{p'}{q'}} < \infty,
	\end{align*}
	which concludes the proof.
\end{proof}

The first step to checking the right-most claim on the line \eqref{eq:goal} is to prove the following Proposition \ref{prop:err:sprs}, which is interesting by itself.

\begin{prop}\label{prop:err:sprs}
	Let $1<p\leq q<\infty,$ $\mu\in A_{p,p},$ $\lambda\in A_{q,q}$ and $b\in\VMO_{\nu}^{\alpha}.$ Let  $\scrS$ be sparse and denote
	\begin{align}\label{eq:split:sparse}
			\scrS_{k} = \scrS\setminus \scrS_{k}^0 ,\qquad  \scrS_k^0 = \Big\{ Q\in\scrS: \ell(Q)\in [k^{-1},k] ,\ \dist(Q,0)\leq k \Big\},\qquad k>0.
	\end{align}
	Then, given any $\varepsilon>0,$ there exists a large $k=k_{\varepsilon}>0$ such that  
	$$
 		\bNorm{ \mathcal{A}_{b,\mathscr{S}_k}f}{L^q_{\lambda}} +	\bNorm{ \mathcal{A}^*_{b,\mathscr{S}_k}f}{L^q_{\lambda}}
 		\leq\varepsilon\Norm{f}{L^p_{\mu}}.
	$$
\end{prop}

\begin{proof} 
We only show the claim for $\mathcal{A}^*_{b,\mathscr{S}_k}.$
	We cover $\scrS_k$ with the union of 
\begin{align*}
	\scrS_{k}^d   &=  \big\{ Q\in	\scrS : \dist(Q,0)> k \big\},\\
	\scrS_{k}^s   &=  \big\{ Q\in	\scrS : \ell(Q)<k \big\}, \\ 
	\scrS_{k}^b   &=  \big\{ Q\in	\scrS : \ell(Q)>k \big\}, 
\end{align*}
and estimate 
\begin{align*}
		\bNorm{ \mathcal{A}^*_{b,\mathscr{S}_k}f}{L^q_{\lambda}} \leq \sum_{w=s,b,d}\bNorm{ \mathcal{A}^*_{b,	\scrS_{k}^w}f}{L^q_{\lambda}}.
\end{align*}

We first consider $w\in\{s,d\}$.
By conditions \eqref{eq:VMOv1} and \eqref{eq:VMOv3} (also, see Remark \ref{rem:vmo3eq}), let $k$ be so large that if $\ell(Q) < k$ or  $\dist(Q,0)>k,$ then $\calO_{\nu}^{\alpha}(b;Q)\leq \varepsilon.$
Especially, if $P\subset Q\in \scrS_{k}^w,$ then either $\ell(P)<k$ or $\dist(P,0)> k,$ depending on $w= s,d.$ 
In any case, following the proof of Theorem \ref{thm:ub}, the factor $\Norm{b}{\BMO_{\nu}^{\alpha}}$ in the estimate \eqref{eq:sub:sharp1} can be replaced with $\varepsilon,$ giving  
	\begin{equation*}\label{eq:midpointclear}
	\begin{split}
		&\bave{\abs{b-\ave{b}_Q}\abs{f}}_Q \\ 
		&\lesssim\varepsilon\Bave{\sum_{\substack{ P\in \widehat{\scrS_{k}^w}\\ P\subset Q}}\frac{\mu^p(P)^{\frac{1}{p}}\lambda^{-q'}(P)^{\frac{1}{q'}}}{\abs{P}}\ave{\abs{f}}_P1_P}_Q \leq	\varepsilon\Bave{\sum_{\substack{ P\in \widehat{\scrS_{k}^w}}}\frac{\mu^p(P)^{\frac{1}{p}}\lambda^{-q'}(P)^{\frac{1}{q'}}}{\abs{P}}\ave{\abs{f}}_P1_P}_Q,
	\end{split}
\end{equation*}	
while the remaining estimate is identical, in total yielding:
\[
 \sum_{w=s,d}\bNorm{ \mathcal{A}^*_{b,	\scrS_{k}^w}f}{L^q_{\lambda}} \lesssim	\varepsilon\Norm{f}{L^p_{\mu}}.
\]

Then, we turn to the case $\scrS_k^b$ of big cubes. Without loss of generality we may assume that $\scrS\subset\scrD.$
By conditions \eqref{eq:VMOv2} and \eqref{eq:VMOv3}, we find cubes $R_i\in\scrD$, where $i=1,\dots,2^d$, such that if 
$\ell(P)\geq \ell(R_1),$ or  $P\subset \R^d\setminus \cup_{i=1}^{2^d} R_i,$ then  $\calO_{\nu}^{\alpha}(b;P)\leq \varepsilon.$
We consider these cubes henceforth fixed throughout the rest of the argument.

We cover
\begin{align}
	\scrS_k^b \subset \cup_{i=1}^{2^d}\scrS_{k,i}^b \cup \big( 	\scrS_k^b \setminus \cup_{i=1}^{2^d}\scrS_{k,i}^b  \big),\qquad \scrS_{k,i}^b = \big\{ Q\in 	\scrS_k^b : Q\cap R_i\not=\emptyset\big\}
\end{align}
and estimate
\begin{align*}
	\bNorm{ \mathcal{A}^*_{b,\mathscr{S}_{k,i}^b}f}{L^q_{\lambda}} \lesssim \sum_{i=1}^{2^d}\bNorm{ \mathcal{A}^*_{b,\scrS_{k,i}^b}f}{L^q_{\lambda}} + \bNorm{ \mathcal{A}^*_{b}\big(f;	\scrS_k^b \setminus \cup_{i=1}^{2^d}\scrS_{k,i}^b \big)}{L^q_{\lambda}}.
\end{align*}
The right-most term is handled as the case of $\scrS_k^d.$ 
For each fixed $i,$ we estimate 
\begin{align*}
	 \bNorm{ \mathcal{A}^*_{b,\mathscr{S}_{k,i}^b}f}{L^q_{\lambda}} \lesssim \BNorm{	\sum_{Q\in\scrS_{k,i}^b}\Bave{\sum_{\substack{P\in\widehat{\mathscr{S}_{k,i}^b}\\ P\subset Q}}\bave{\abs{b-\ave{b}_P}}_P\ave{\abs{f}}_P1_P}_Q1_Q}{L^q_{\lambda}}
\end{align*}
and further split and estimate the interior sum as 
\begin{align}
	\sum_{\substack{P\in\widehat{\mathscr{S}_{k,i}^b}\\ P\subset Q}} 
	\leq \sum_{j=1}^{2^d}\Big(\sum_{\substack{P\in\widehat{\mathscr{S}_{k,i}^b}\\ P \subset R_j }} 
	+ \sum_{\substack{P\in\widehat{\mathscr{S}_{k,i}^b}\\ P \supset R_j}}\Big) 
	+  \sum_{\substack{P\in\widehat{\mathscr{S}_{k,i}^b}\\ P\subset  \R^d\setminus \cup_{j=1}^{2^d} R_j }} = \sum_{j=1}^{2^d}\big(I_{i,j} + II_{i,j}\big) + III_i.
\end{align}
The term corresponding to the sum $III_i$ is estimated as in the case $\scrS_k^d$ of the distant cubes; while for $II_{i,j},$ we have $R_j\subset P,$ hence $\ell(P)\geq\ell(R_1),$ hence $\calO_{\nu}^{\alpha}(b;P)\leq \varepsilon$ and the rest of the estimate is as before.
It remains to handle the term corresponding to $I_{i,j},$ where we rewrite the condition $Q\cap R_i\neq\emptyset$ as $Q\supset R_i$ under the assumption that $Q\in\scrD_k^b$ is bigger than $R_i$:
\begin{align}\label{xx}
\BNorm{	\sum_{\substack{Q\in\scrS_{k}^b \\ Q\supset R_i}}\Bave{\sum_{\substack{P\in\widehat{\mathscr{S}_{k,i}^b}\\ P\subset R_j}}\bave{\abs{b-\ave{b}_P}}_P\ave{\abs{f}}_P1_P}_Q1_Q}{L^q_{\lambda}}.
\end{align}
We estimate the interior of \eqref{xx} as 
\begin{align}\label{xxx}	
\Bave{\sum_{\substack{P\in\widehat{\mathscr{S}_{k,i}^b}\\ P\subset R_j}}\bave{\abs{b-\ave{b}_P}}_P\ave{\abs{f}}_P1_P}_Q\leq \Norm{b}{\BMO_{\nu}^{\alpha}}\Bave{\sum_{\substack{P\in\widehat{\mathscr{S}_{k,i}^b}\\ P\subset R_j}}\frac{\mu^p(P)^{\frac{1}{p}}\lambda^{-q'}(P)^{\frac{1}{q'}}}{\abs{P}}\ave{\abs{f}}_P1_P}_Q,
\end{align}
and further
\begin{align*} 
	&\Bave{\sum_{\substack{P\in\widehat{\mathscr{S}_{k,i}^b}\\ P\subset R_j}}\frac{\mu^p(P)^{\frac{1}{p}}\lambda^{-q'}(P)^{\frac{1}{q'}}}{\abs{P}}\ave{\abs{f}}_P1_P}_Q \\ 
	&\qquad=\frac{\abs{R_j}}{\abs{Q}}\frac{1}{\Norm{1_{R_j}}{L^q_{\lambda}}}\BNorm{\mathcal{A}_{\{R_j\}}\big(\mathcal{A}_{\mu,\lambda}^{p,q}(f;\widehat{\scrS}_{k,i}^b)\big)}{L^{q}_{\lambda}} \lesssim \frac{\abs{R_j}}{\abs{Q}}\frac{\Norm{f}{L^{p}_{\mu}}}{\Norm{1_{R_j}}{L^q_{\lambda}}}\lesssim \frac{\abs{R_i}}{\abs{Q}}\Norm{f}{L^{p}_{\mu}},
\end{align*}
where in the last estimate we used the estimates $\ell(R_i)\sim\ell(R_j),$  and $\Norm{1_{R_j}}{L^q_{\lambda}}\sim 1,$ recalling that the boundedly many cubes $R_i$ are considered fixed, and hence dependence on them will be suppressed.
Substituting, we find
\begin{align}\label{xxxx}
	\eqref{xx} \lesssim \Norm{b}{\BMO_{\nu}^{\alpha}} \Norm{f}{L^{p}_{\mu}}\BNorm{\sum_{\substack{Q\in\scrS_{k}^b \\ Q\supset R_i}} \frac{\abs{R_i}}{\abs{Q}} 1_Q}{L^q_{\lambda}}. 
\end{align}
It remains to show that the last term can be made small.
Let $\delta>0.$ Provided that $k = k_{\delta}>>\ell(R_i)$ is taken sufficiently large, by summing a geometric series we find
\[
	\sum_{\substack{Q\in\scrS_{k}^b \\ Q\supset R_i}} \frac{1_Q}{\abs{Q}} \abs{R_i}\leq \delta.
\]
On the other hand
\[
	\sum_{\substack{Q\in\scrS_{k}^b \\ Q\supset R_i}} \frac{1_Q}{\abs{Q}}\abs{R_i} \leq \sum_{\substack{Q\in\scrD \\ Q \supset R_i}} \frac{\abs{R_i}}{\abs{Q}} 1_Q \lesssim M(1_{R_i}). 
\]
Hence, we find 
\begin{align}\label{xxxxx}
	\BNorm{\sum_{\substack{Q\in\scrS_{k}^b \\ Q\supset R_i}} \frac{1_Q}{\abs{Q}} \abs{R_j}}{L^q_{\lambda}} \lesssim  \bNorm{\min(\delta,M(1_{R_i}))}{L^q_{\lambda}}.
\end{align}
Since $M$ is a bounded operator on $L^q_{\lambda}$ (by $\lambda\in A_{q,q}$), by dominated convergence the right-hand side of \eqref{xxxxx} can be made smaller than $\varepsilon,$ by choosing $\delta$ small ($k_{\delta}$ large) enough.
\end{proof}

\begin{cor}	Let $1<p\leq q<\infty,$ $\mu\in A_{p,p},$ $\lambda\in A_{q,q}$ and $b\in\VMO_{\nu}^{\alpha}.$ Let $\scrS$ be sparse. Then, 
\begin{align*}
	[b,A_{\scrS}], A_{b,\scrS}\in\calK(L^p_{\mu},L^q_{\lambda}).
\end{align*}
\end{cor}
\begin{proof} By Proposition \ref{prop:err:sprs}, we can approximate both by finite rank operators to arbitrary precision.
\end{proof}

\begin{prop}Suppose that $b\in\VMO_{\nu}^{\alpha}.$ Then,
	\begin{align*}
		\lim_{\varepsilon\to 0}\bNorm{[b,	 T_{\varepsilon}]}{L^{p}_{\mu}\to L^q_{\lambda}} = 0.
	\end{align*}
\end{prop}
\begin{proof} By Proposition \ref{prop:err:sprs} it is enough to show that for each $k>0,$ the following holds: for all $\varepsilon > 0$ small enough (depending on $k$) there exists a sparse collection $\scrS$ such that
	\begin{align}\label{eq:step1}
		\babs{\bave{[b,T_{\varepsilon}]f , g}}\lesssim 	\babs{ \bave{ \mathcal{A}_{b,\mathscr{S}}\abs{f} , \abs{g}}} + \babs{ \bave{  \mathcal{A}^*_{b,\mathscr{S}}\abs{f} , \abs{g}}},\qquad \scrS = \scrS_k,
	\end{align}
where $\scrS_k$ is as in line \eqref{eq:split:sparse}. Recall that $\varepsilon = r = R^{-1}.$ As 
\begin{align}\label{eq:step2}
	T_{\varepsilon} = 	\varphi^{0,R}T	\varphi^{0,R}\varphi_x^{0,r} + \Big[  \varphi^{R,\infty}T \varphi^{R,\infty}
	+ \varphi^{R,\infty}T \varphi^{0,R}
	+ \varphi^{0,R}T \varphi^{R,\infty} \Big],
\end{align}
it is enough to show that claim \eqref{eq:step1} is satisfied for each of the four pieces of $T_{\varepsilon}.$ Each of the three pieces inside the brackets is handled in the same way. For example, directly from the sparse domination of $[b,T]$ we acquire a sparse collection $\scrS$ such that
\begin{align}\label{step1}
	\babs{\bave{ [b, \varphi^{0,R}T \varphi^{R,\infty}]f ,g}} \lesssim  \babs{ \bave{ \mathcal{A}_{b,\mathscr{S}}\abs{f\varphi^{R,\infty}} , \abs{g\varphi^{0,R}}}} + \babs{ \bave{  \mathcal{A}^*_{b,\mathscr{S}}\abs{f\varphi^{R,\infty}} , \abs{g\varphi^{0,R}}}}.
\end{align}
Considering the first term on the right-hand side of \eqref{step1}, we have 
\begin{align*}
 \babs{ \bave{ \mathcal{A}_{b,\mathscr{S}}\abs{f\varphi^{R,\infty}} , \abs{g\varphi^{0,R}}}} \leq \babs{ \bave{ \mathcal{A}_{b}\big( \abs{f}; \mathscr{S}^{R,\infty}\big) , \abs{g}}},
\end{align*}
where $\mathscr{S}^{R,\infty} = \{ Q\in\scrS: Q\cap B(0,\frac{1}{2} R)^c \neq\emptyset \}.$ Clearly $\mathscr{S}^{R,\infty} = \scrS_k^{R,\infty}$ for sufficiently small $\varepsilon$ (large $R$).
The other bracketed pieces are handled identically. 

It remains to handle the term $\bave{ [b, \varphi^{0,R}T	\varphi^{0,R}\varphi_x^{0,r}]f ,g},$ which we write as 
\[
\bave{ [b, \varphi^{0,R}T	\varphi^{0,R}\varphi_x^{0,r}]f ,g} = \bave{ [b, \wt{T}_r]\widetilde{f} ,\widetilde{g}},\qquad  \widetilde{T}_r = T	\varphi_x^{0,r},\quad \wt{f} = \varphi^{0,R} f,\quad \wt{g} = \varphi^{0,R} g.
\]  
We express $\R^d = \bigcup_j Q_j$ as a disjoint union of cubes with $\diam(Q_j) = r.$
Then, we have 
\begin{align*}
\bave{ [b, \wt{T}_r]\widetilde{f} ,\widetilde{g}} =  \sum_j \bave{ 1_{Q_j}[b, \wt{T}_r]\widetilde{f} ,\widetilde{g}}
\end{align*}
and it is enough to show that each $\bave{ 1_{Q_j}[b, \wt{T}_r]\widetilde{f} ,\widetilde{g}}$ admits a sparse domination localized to the cube $Q_j.$ Using 
\[
1_{Q_j}(x)\wt{T}_rf(x) = 1_{Q_j}(x)\wt{T}_r(1_{Q_j^*}f)(x),\qquad Q_j^* = 3Q_j,
\]
we write
\begin{align*}
\bave{ 1_{Q_j}[b, \wt{T}_r]\widetilde{f} ,\widetilde{g}} = \bave{ 1_{Q_j}[b, \wt{T}_r](1_{Q_j^*}\widetilde{f}) ,\widetilde{g}}.
\end{align*}
Next we argue that $\widetilde{T}_r$ is a CZO with $\Norm{\wt{T}_r}{\CZO}\lesssim 1,$ with constant independent of $r>0.$ Since
\[
\wt{T}_r = T- T\varphi_x^{r,\infty},
\] 
for the uniform $L^2$-to-$L^2$-boundedness it is enough to show that the $T\varphi_x^{r,\infty}$ are uniformly $L^2$-to-$L^2$-bounded.  By Cotlar's inequality, the truncated operators $T_rf(x) = T(1_{B(x,r)^c}f)(x)$ are uniformly bounded and hence it is enough to give the following uniform bound for the difference:
\begin{align*}
	\babs{\big(T_r - T\varphi_x^{r,\infty}\big)f(x)} \leq \int_{B(x.r)\setminus B(x,\frac{r}{2})}\abs{K(x,y)}(1-\varphi_x^{r,\infty}(y))\abs{f(y)}\ud y \lesssim Mf(x).
\end{align*}
Next we check that $\varphi_x^{0,r}K(x,y)$ is a CZ-kernel with uniform constants. The size estimate is immediate by $\abs{\varphi_x^{0,r}K(x,y)}\leq \Norm{\varphi^{0,r}}{L^{\infty}}\abs{K(x,y)}.$  For the regularity estimate, provided that $\abs{x-x'}\leq \frac{1}{2}\abs{x-y},$ we have 
\begin{equation}\label{ghj}
	\begin{split}
		&\babs{\varphi_{x'}^{0,r}(y)K(x',y)-\varphi_x^{0,r}(y)K(x,y)} \\ 
	& \leq \babs{\varphi_{x'}^{0,r}(y)\big(K(x',y)-K(x,y)\big)}+ \babs{\big(\varphi_{x'}^{0,r}(y)-\varphi_x^{0,r}(y)\big)K(x,y)} \\ 
	& \lesssim \Norm{\varphi_{x'}^{0,r}}{L^{\infty}}\omega\big( \frac{\abs{x-x'}}{\abs{x-y}} \big)\abs{x-y}^{-d} + \babs{\varphi_{x'}^{0,r}(y)-\varphi_x^{0,r}(y)}\abs{x-y}^{-d}.
	\end{split}
\end{equation}
The remaining term we estimate, using $\abs{x-x'}\leq \frac{1}{2}\abs{x-y}$ in the  first identity, as 
\begin{align*}
	\abs{\varphi_{x'}^{0,r}(y)-\varphi_x^{0,r}(y)}& = 	\abs{\varphi_{x'}^{0,r}(y)-\varphi_x^{0,r}(y)}1_{B(x,2r)}(y) \\ 
	&\lesssim \Norm{\nabla\varphi^{0,r}}{L^{\infty}}\abs{x-x'}1_{B(x,2r)}(y)\sim \frac{1}{r}\abs{x-x'}1_{B(x,2r)}(y)\lesssim \frac{\abs{x-x'}}{\abs{x-y}}.
\end{align*}
 The other estimate is completely symmetric. We have shown that  $\Norm{\wt{T}_r}{\CZO}\lesssim 1,$ independently of $r>0.$ 

Now, by the standard proof of the sparse domination of the commutator, see e.g. \cite{LOR1}, we conclude that
\begin{align}\label{eq:step3}
	\babs{\bave{ 1_{Q_j}\big[b, \widetilde{T}_r\big](1_{Q_j^*}\widetilde{f}) ,\widetilde{g}}} \lesssim  \bave{ \mathcal{A}_{b}\big( \abs{f}; \mathscr{S}(Q_j) \big), \abs{g}} + \bave{  \mathcal{A}^*_{b}\big( \abs{f}; \mathscr{S}(Q_j) \big) , \abs{g}},
\end{align}
where $\scrS(Q_j)$ is a sparse collection inside $Q_j^*,$ and the implicit constant in \eqref{eq:step3} is independent of $r.$ All in all, we have shown that
\begin{align*}
	\babs{	\bave{ [b, \varphi^{0,R}T	\varphi^{0,R}\varphi_x^{0,r}]f ,g}} \lesssim \bave{ \mathcal{A}_{b}\big( \abs{f}; \mathscr{S} \big), \abs{g}} + \bave{  \mathcal{A}^*_{b}\big( \abs{f}; \mathscr{S} \big) , \abs{g}},
\end{align*}
where $\mathscr{S} =  \cup_j\mathscr{S}(Q_j).$ By $\ell(Q_j) = r,$ for all $j,$ it follows that $\mathscr{S} = \mathscr{S}_k$ for a choice of $\varepsilon = r$ sufficiently small. 
\end{proof}

	\subsection{Necessity} 
	Let $T$ be a non-degenerate singular integral. In this section we show that $[b,T]\in\calK(L^p_{\mu}, L^q_{\lambda})$ implies $b\in\VMO_{\nu}^{\alpha}.$
	
The following Lemma is originally implied in Uchiyama \cite{Uch1978} and also in several later works.
\begin{lem}\label{lem:seq}
	Let $1<p,q<\infty$, let $\mu$ and $\lambda$ be arbitrary weights, and let $U\colon L_\mu^p\to L_\lambda^q$ be a bounded linear operator. Then, there does not exist a sequence $\{u_i\}=\{u_i\}_{i=1}^\infty$ with the following properties:
	\begin{enumerate}[$(i)$]
		\item $\sup_{i\in\N}\|u_i\|_{L_\mu^p}\lesssim 1,$
		\item $\{x\colon u_i(x)\neq 0\}\cap \{x\colon u_j(x)\neq 0\}=\emptyset$, whenever $i\neq j,$ and
		\item there exists $\Phi\in L^q_{\lambda}$ such that  $$\lim_{i\to\infty} \|\Phi-U(u_i)\|_{L_\lambda^q} = 0,\qquad\|\Phi\|_{L_\lambda^q}>0.$$
	\end{enumerate} 
\end{lem}

\begin{proof} We show that the existence of such a sequence contradicts the boundedness of $U.$
	By the point $(iii),$ and by passing to a subsequence if needed, we can assume that $\|\Phi-U(u_i)\|_{L_\lambda^q}\leq 2^{-i}.$
	By $p,q'>1,$ let us choose a sequence $\{a_i\} \in \big(\ell^p\cap\ell^{q'}\big)\setminus \ell^1$ of positive numbers. We define $g_k=\sum_{i=1}^ka_iu_i$ and next show that $\{g_k\}$ is a Cauchy sequence in $L_\mu^p$. Suppose that $m\leq k.$ By assumptions $(i)$ and $(ii)$ we find that 
	\begin{align*}
		\|g_k-g_m\|_{L_\mu^p} =\big\|\sum_{i=m+1}^ka_iu_i\big\|_{L_\mu^p}
		=\bigg(\sum_{i=m+1}^k|a_i|^p\|u_i\|_{L_\mu^p}^p\bigg)^\frac{1}{p} \lesssim \bigg(\sum_{i=m+1}^k|a_i|^p\bigg)^\frac{1}{p},
	\end{align*}
	from which by $\{a_i\}_i\in\ell^p$ we see that $\{g_k\}$ is Cauchy. By completeness of $L_\mu^p,$ the sequence of functions $g_k$ converges in $L_\mu^p;$ and since $U$ is continuous, $U(g_k)$ converges in $L_\lambda^q.$  	
	Let us define $h_k:=\sum_{i=1}^ka_i\Phi.$ Then, we have 
	\begin{align*}
		\bNorm{h_k - U(g_k)}{L^q_{\lambda}} &\leq \sum_{i=1}^\infty|a_i|\|\Phi-U(u_i)\|_{L_\lambda^q} \\
		&\leq \|\{a_i\}\|_{\ell^{q'}}\big(\sum_{i=1}^\infty\|\Phi-U(u_i)\|_{L_\lambda^q}^q\big)^\frac{1}{q} \leq \|\{a_i\}\|_{\ell^{q'}}\big(\sum_{i=1}^\infty2^{-iq}\big)^\frac{1}{q} <\infty
	\end{align*}
uniformly in $k,$ and as $U(g_k)$ converges in $L^q_{\lambda},$ it follows
that $\sup_{k\in\N}\Norm{h_k}{L^q_{\lambda}} \lesssim 1$ and
\begin{align*}
	1\gtrsim \sup_{k\in\N}\Norm{h_k}{L^q_{\lambda}}  =\sup_{k\in\N}\sum_{i=1}^k|a_i|\|\Phi\|_{L_\lambda^q} \sim \Norm{\{a_i\}_i}{\ell^1},
\end{align*}
which contradicts $\{a_i\}\not\in\ell^1.$
\end{proof}

\begin{lem}\label{lem:comp:nec1} Let $\nu\in A_{\infty}$ and suppose that $b\in \BMO_{\nu}^{\alpha}\setminus\VMO_{\nu}^{\alpha}$ with $\alpha>-d.$ Then, there exists a sparse collection of cubes $\mathscr{S} = \{Q,E_Q\}_{Q\in\mathscr{S}}$ such that $\calO_{\nu}^{\alpha}(b;E_Q) \gtrsim 1,$ uniformly.
	\end{lem}
	\begin{proof}  
		If the $\VMO_{\nu}^{\alpha}$ condition fails via \eqref{eq:VMOv3}, then the construction of $\mathscr{S}$ is immediate; we can guarantee that all the cubes are disjoint by choosing cubes ever farther away from the origin and we simply let $E_Q = Q.$
		
		For the duration of this proof let us denote $\beta = 1+\alpha/d>0.$
		If the $\VMO_{\nu}^{\alpha}$ condition fails via \eqref{eq:VMOv1}, then we are guaranteed a sequence of cubes $\{Q_j\}_j$ such that 
		\[
		\calO_{\nu}^{\alpha}(b;Q_j)\gtrsim 1,\qquad \lim_{j\to\infty}\ell(Q_j)=0.
		\]
		By passing to a subsequence, we may assume that $\sum_{j=1}^\infty\abs{Q_j}<\infty$.
		
In order to construct the disjoint subsets sets $E_j=E_{Q_j}$, we make the following observation: 
\begin{equation}\label{eq:calOEhalf}
  \forall Q\quad \exists \theta_Q\in(0,\frac12)\quad \forall E\subset Q:\quad \abs{E}\geq(1-\theta_Q)\abs{Q}\quad\Rightarrow\quad
  \calO_{\nu}^{\alpha}(b;E)>\frac12  \calO_{\nu}^{\alpha}(b;Q).
\end{equation}
Indeed, if not, then we can find a sequence of $E_n\subset Q$ with $\calO_{\nu}^{\alpha}(b;E_n)\leq\frac12\calO_{\nu}^{\alpha}(b;Q)$, while $\abs{E_n}\to\abs{Q}$, and a subsequence will satisfy $1_{E_n}\to 1_Q$ almost everywhere. Then dominated convergence shows that $\nu(E_n)=\int 1_{E_n}\nu\to \nu(Q)$ and $\ave{b}_{E_n}=\abs{E_n}^{-1}\int_Q 1_{E_n}b\to\ave{b}_Q$,
and finally
\begin{equation*}
  \int_{E_n}\abs{b-\ave{b}_{E_n}}=\int_Q 1_{E_n}\abs{b-\ave{b}_Q+(\ave{b}_Q-\ave{b}_{E_n})}\to \int_Q\abs{b-\ave{b}_Q},
\end{equation*}
so that $\calO_{\nu}^\alpha(b;E_n)\to \calO_{\nu}^\alpha(b;Q)$, contradicting $\calO_{\nu}^{\alpha}(b;E_n)\leq\frac12\calO_{\nu}^{\alpha}(b;Q)$, thus proving \eqref{eq:calOEhalf}.

We now pick a subsequence $\{Q_j'=Q_{k(j)}\}_j$ of the original sequence $\{Q_j\}_j$ and as follows. Let $Q_1'=Q_1$. Assuming that $Q_1',\ldots,Q_j'=Q_{k(j)}$ have already been chosen, we choose $k(j+1)>k(j)$ large enough so that
\begin{equation*}
  \sum_{i=k(j+1)}^\infty\abs{Q_i}\leq\theta_j\abs{Q_j'},\qquad\theta_j=\theta_{Q_j'},
\end{equation*}
and then set $Q_{j+1}'=Q_{k(j+1)}$. For the sequence thus constructed, we hence have
\begin{equation*}
  \sum_{i=j+1}^\infty\abs{Q_i'}\leq\sum_{i=k(j+1)}^\infty\abs{Q_i}\leq\theta_j\abs{Q_j'},
\end{equation*}
and thus $E_j=Q_j'\setminus\bigcup_{i=j+1}^\infty Q_i'$ satisfies $\abs{E_j}\geq(1-\theta_{Q_j'})\abs{Q_j}$. By \eqref{eq:calOEhalf}, this implies that $\calO_{\nu}^\alpha(b;E_j)\geq\frac12\calO_{\nu}^\alpha(b;Q_j')\gtrsim 1$. This gives the required sparse collection $\{Q,E_Q\}_{Q\in\mathscr S}=\{Q_j',E_j\}_{j=1}^\infty$ in this case.
		
%

		Lastly, suppose that the $\VMO_{\nu}^{\alpha}$ condition fails via \eqref{eq:VMOv2}, i.e. we find a sequence of cubes $\{Q_j\}_{j=0}^{\infty}$ such that 
		\[
		\calO_{\nu}^{\alpha}(b;Q_j)\gtrsim 1,\qquad \lim_{j\to\infty}\ell(Q_j)=\infty.
		\]
	We only show first step of the iterative construction of $\mathscr{S},$ as the subsequent inductive steps are entirely analogous. We begin with setting $\mathscr{S}_0 = \{Q_0,E_{Q_0}\}$ with $E_{Q_0} = Q_0$ and show how to construct $\mathscr{S}_1\supset \mathscr{S}_0.$ Let us denote $P_k=Q_j$ for some $j$ such that $\ell(Q_j)\geq 2^k$ and $\ell(P_k)\geq \ell(Q_0).$ If $P_k\cap Q_0=\emptyset,$ we choose $Q_1 = P_k$ and $E_{Q_1} = Q_1$ and with $\mathscr{S}_1 = \{Q_i,E_{Q_i}\}_{i=0,1}$ we are done. If $P_k\cap Q_0\not=\emptyset,$ by $\ell(P_k)\geq \ell(Q_0)$ there exists a cube $\widehat{P}_k\supset P_k\cup Q_0$ such that $\ell(P_k)\sim\ell(\widehat{P}_k);$ and by $\nu$ being doubling, we know that $\calO_{\nu}^{\alpha}(\widehat{P}_k)\gtrsim \calO_{\nu}^{\alpha}(P_k) \gtrsim 1.$ Hence, without loss of generality, let us denote $P_k = \widehat{P}_k$ and assume that $Q_0\subset P_k.$ We will next show that if $k$ is sufficiently large, then $E_{P_k} = P_k\setminus Q_0$ satisfies  $\calO_{\nu}^{\alpha}(b;E_{P_k}) \gtrsim 1,$ and then we set $Q_1 = P_k$ and $E_{Q_1} = E_{P_k}.$
	
	We first estimate
	\begin{align}\label{eq:seq0}
		\int_{P_k}\abs{b-\ave{b}_{E_{P_k}}} \leq \int_{Q_0}\abs{b-\ave{b}_{Q_0}} +  \int_{Q_0}\abs{\ave{b}_{Q_0}-\ave{b}_{E_{P_k}}} +\int_{E_{P_k}}\abs{b-\ave{b}_{E_{P_k}}},
	\end{align}
where the last term is of the desired form. We will next show that 
\begin{align}\label{eq:seq1}
		\lim_{k\to\infty} 	\nu(P_k)^{-\beta}\int_{Q_0}\abs{b-\ave{b}_{Q_0}}=0,\qquad  	\lim_{k\to\infty}	\nu(P_k)^{-\beta}\int_{Q_0}\abs{\ave{b}_{Q_0}-\ave{b}_{E_{P_k}}} = 0.
\end{align}
The left claim on line \eqref{eq:seq1} is immediate from $b\in L^1_{\loc}$ and $\nu\in A_{\infty}$ and we provide details only for the right limit.
	Let $\{R_k\}_{k=0}^{m}$ (we have $m\sim\log_2(\abs{P_k}/\abs{Q_0})$ but the exact value of $m$ will not play a further role in the argument) be a sequence of cubes such that
	\[
	R_0 = Q_0,\qquad R_m = P_k,\qquad R_k\subset R_{k+1},\qquad \ell(R_{k+1})\sim \ell(R_k),
	\] 
also denote $R_{m+1} = E_{P_k}.$ Then, we estimate
	\begin{align}
		\abs{\ave{b}_{Q_0}-\ave{b}_{E_{P_k}}}\leq \sum_{j=0}^{m} \abs{\ave{b}_{R_{j+1}}-\ave{b}_{R_{j}}}\lesssim \Norm{b}{\BMO_{\nu}^{\alpha}} \sum_{j=0}^{m} \frac{\nu(R_j)^{\beta}}{\abs{R_j}}
	\end{align}
	and hence we find that 
	\begin{align}\label{eq:seq2}
	\nu(P_k)^{-\beta}\int_{Q_0}\abs{\ave{b}_{Q_0}-\ave{b}_{E_{P_k}}}\lesssim \Norm{b}{\BMO_{\nu}^{\alpha}}  \sum_{j=0}^{m} \frac{\abs{Q_0}}{\abs{R_j}}\Big(\frac{\nu(R_j)}{\nu(P_k)}\Big)^{\beta}.
	\end{align}
Next we show that the sum tends to zero as $k$ (equivalently $m$) tends to infinity.
Since $\nu\in A_{\infty}$ we know that for some $\delta>0$ there holds that  $\nu(E_Q)/\nu(Q)\lesssim \big(\abs{E_Q}/\abs{Q}\big)^{\delta}$ for all $E_Q\subset Q.$ Hence,
\begin{align*}
	 \sum_{j=0}^{m} \frac{\abs{Q_0}}{\abs{R_j}}\Big(\frac{\nu(R_j)}{\nu(P_k)}\Big)^{\beta} &\lesssim \sum_{j=0}^{m} \frac{\abs{Q_0}}{\abs{R_j}}\Big(\frac{\abs{R_j}}{\abs{P_k}}\Big)^{\delta\beta}\lesssim \sum_{j=0}^{m} 2^{-jd} 2^{-(m-j)d\delta\beta}\\ 
	 &\lesssim  \sum_{j=0}^{m} 2^{-\max\{j,m-j\} \min(1,\delta\beta)d}
\end{align*}
and clearly the right-hand side tends to zero by $\beta>0$; we conclude the right limit on line \eqref{eq:seq1}.
All in all, using
\begin{equation*}
  \int_{P_k}\abs{b-\ave{b}_{P_k}}
  =\int_{P_k}\abs{b-c-\ave{b-c}_{P_k}}
  \leq 2\int_{P_k}\abs{b-c}
\end{equation*}
for any constant $c$, and then lines \eqref{eq:seq0} and \eqref{eq:seq1}, we have shown that for any $\varepsilon>0$ there exists $k$ large enough so that
\begin{equation}\label{eq:seq5}
\begin{split}
		1&\lesssim 	\calO_{\nu}^{\alpha}(b;P_k) =\nu(P_k)^{-\beta}\int_{P_k}\abs{b-\ave{b}_{P_k}} \lesssim	\nu(P_k)^{-\beta}\int_{P_k}\abs{b-\ave{b}_{E_{P_k}}} \\ 
			&\lesssim \varepsilon + \nu(P_k)^{-\beta}\int_{E_{P_k}}\abs{b-\ave{b}_{E_{P_k}}} \\ 
			& \lesssim \varepsilon + 	\nu(E_{P_k})^{-\beta}\int_{E_{P_k}}\abs{b-\ave{b}_{E_{P_k}}} = \varepsilon + \calO_{\nu}^{\alpha}(b;E_{P_k}).
\end{split}
\end{equation}
Choosing $\varepsilon>0$ small enough, we conclude that $\calO_{\nu}^{\alpha}(b;E_{P_k}) \gtrsim 1.$
	\end{proof}

\begin{prop}\label{prop:comp:nec}
	Let $1<p,q<\infty,$ $\mu\in A_{p,p}$ and $\lambda\in A_{q,q}$. Let $\nu$ be the fractional Bloom weight of Definition \ref{defn:bloom:weight}, let $T$ be a non-degenerate SIO and $b\in L_\mathrm{loc}^1(\R^d;\C)$. Then, $[b,T]\in \mathcal{K}(L^p_{\mu}, L^q_{\lambda})$ implies that $b\in \VMO_{\nu}^\alpha(\R^d)$.
\end{prop}
\begin{proof}
As the commutator is compact, it is in particular bounded, and by Proposition \ref{prop:bdd:lb} we have $b\in \BMO_{\nu}^\alpha$.
	Let us assume for a contradiction that $b\notin\VMO_{\nu}^{\alpha}$,  i.e. $b\in\BMO_{\nu}^\alpha\setminus\VMO_{\nu}^{\alpha},$ and by Proposition \ref{lem:comp:nec1} find a sparse sequence $\{Q_j, E_j\}_{j=1}^\infty$ of cubes such that 
 $\mathcal{O}_\nu^\alpha(b;E_j) \gtrsim 1.$ By Proposition \ref{prop:osc1} applied with $g_{E_j} = 1_{E_j}$ and $g_{\wt{E}} = g_{\wt{Q}_j} = 1_{\wt{Q}_j}$ we have
	\begin{equation}\label{eq:necessarycompactnesswaf}
	\begin{split}
		\int_{E_j}|b-\langle b\rangle_{E_j}| 
		&\lesssim \big|\big\langle [b,T]g_{E_j}, h_{\widetilde{Q}_j}\big\rangle\big|+\big|\big\langle [b,T]h_{E_j},g_{\widetilde{Q}_j}\big\rangle\big| \\
		&\sim \max\Big\{\big|\langle [b,T]g_{E_j}, h_{\widetilde{Q}_j}\rangle\big|,\big|\big\langle [b,T]h_{E_j},g_{\widetilde{Q}_j}\big\rangle\big|\Big\} = \big|\big\langle [b,T]\phi_{E_j}, \psi_{\wt{Q}_j}\big\rangle\big|,
	\end{split}
	\end{equation}
	 for the choice of $\phi_{E_j} \in \{g_{E_j},h_{E_j}\}$ and $\psi_{\wt{Q}_j}\in \{ h_{\wt{Q}_j}, g_{\wt{Q}_j}\}$ that achieve the maximum. By Proposition \ref{prop:osc1} we have
	 \begin{align}
	\abs{ \phi_{E_j}}\lesssim 1_{E_j},\qquad \abs{\psi_{\wt{Q}_j}}\lesssim 1_{\wt{Q}_j}.
	 \end{align}
	Recall that the implicit constants above are independent of $j$ (as are the implicit constants later on in this proof as well).
	Combining the above information, we have
\begin{equation}\label{eq:uEj}
		\begin{split}
		1 &\lesssim \mathcal{O}_\nu^\alpha(b;E_j) = 	\frac{1}{\nu(E_j)^{1+\alpha/d}}\int_{E_j}|b-\langle b\rangle_{E_j}| \\ 
		&\lesssim \frac{1}{\nu(E_j)^{1+\alpha/d}}\big|\big\langle [b,T]\phi_{E_j}, \psi_{\wt{Q}_j}\big\rangle \big|
		\sim \frac{1}{\nu(Q_j)^{1+\alpha/d}}\big|\big\langle [b,T]\phi_{E_j}, \psi_{\wt{Q}_j}\big\rangle\big| \\
		&\lesssim \frac{1}{\mu^p(Q_j)^{1/p}\lambda^{-q'}(Q_j)^{1/q'}}\big|\big\langle [b,T]\phi_{E_j}, \psi_{\wt{Q}_j}\big\rangle\big| =\frac{1}{\lambda^{-q'}(Q_j)^{1/q'}}\big|\big\langle [b,T]u_{E_j}\lambda, \psi_{\wt{Q}_j}\lambda ^{-1}\big\rangle\big| \\
		&\leq \frac{1}{\lambda^{-q'}(Q_j)^{1/q'}}\|[b,T]u_{E_j}\|_{L_\lambda^q} \|\psi_{\wt{Q}_j}\|_{L_{\lambda^{-1}}^{q'}} \\ 
		&\lesssim \frac{\lambda^{-q'}(\widetilde{Q}_j)^\frac{1}{q'}}{\lambda^{-q'}(Q_j)^\frac{1}{q'}}\|[b,T]u_{E_j}\|_{L_\lambda^q} \sim 	\|[b,T]u_{E_j}\|_{L_\lambda^q}
	\end{split}
\end{equation}
	where we denote $u_{E_j}:=\mu^p(Q_j)^{-1/p}\phi_{E_j}.$
By the uniform bound
	\[
	\|u_{E_j}\|_{L_{\mu}^p}^p=\mu^p(Q_j)^{-1}\int_{Q_j}|\phi_{E_j}|^p\mu^p\leq\|\phi_{E_j}\|_{L^\infty}^p\lesssim 1,
	\]
the compactness of $[b,T]$ gives a subsequence of $\big\{[b,T]u_{E_j}\big\}_{j=1}^\infty$ with a limit $\Phi$ in $L_\lambda^q$, and by \eqref{eq:uEj} we have $\|\Phi\|_{L_\lambda^q}>0$. Furthermore, the functions $u_{E_{j_i}}$ are disjointly supported. Concluding, we have constructed a sequence of functions just as in Lemma \ref{lem:seq}, hence $[b,T]$ is not bounded,  a contradiction, and thus necessarily $b\in \VMO_{\nu}^\alpha(\R^d).$
\end{proof}

%
%
%
%

\appendix

\section{Weighted fractional John-Nirenberg inequality for $\VMO$}\label{sect:appB}
Let us define the weighted $\BMO_w^{r,\alpha}$ and $\VMO_{w}^{r,\alpha}\subset \BMO_w^{r,\alpha}$ similarly as in Definition \ref{defn:VMOw} but with the oscillation
\begin{align}
	\calO_w^{r,\alpha}(f;Q) = w(Q)^{-\alpha/d}\left( \frac{1}{w(Q)}\int_Q\big(\abs{f-\ave{f}_Q}\frac{1}{w}\big)^{r}\ud  w\right)^{\frac{1}{r}},
\end{align}
note that $\calO_w^{\alpha}(f;Q)= 	\calO_w^{1,\alpha}(f;Q).$
When $\alpha = 0,$ the following Theorem \ref{appBthm0} is a classical result of Muckenhoupt and Wheeden \cite{MucWhe1976wbmo}.
\begin{thm}\label{appBthm0} Let $1\leq p <\infty$ and suppose that $w\in A_p.$ Let $\alpha\in[0,\infty).$ Then, there holds that 
	\begin{align}\label{appB:thm1}
		\Norm{b}{\BMO_w^{r,\alpha}}\sim \Norm{b}{\BMO_w^{1,\alpha}},
	\end{align}
	whenever $1\leq r \leq p'$ and $r<\infty.$
\end{thm}
Here we give a short proof of Theorem \ref{appBthm0} and extend it to weighted fractional $\VMO$ as

\begin{thm}\label{appBthm1} Let $1\leq p <\infty$ and suppose that $w\in A_p.$ Let $\alpha\in[0,\infty).$ Then, there holds that 
	\begin{align}\label{appB:thm0}
	\VMO_w^{r,\alpha} = \VMO_w^{1,\alpha},
	\end{align}
	whenever $1\leq r \leq p'$ and $r<\infty.$ 
\end{thm}

\begin{rem} If $w\in A_p,$ then by the reverse Hölder property $w\in A_{p-\delta}$ for some $\delta>0.$ Hence, the conclusions \eqref{appB:thm0} and \eqref{appB:thm1} hold with $1\leq r < (p-\delta)',$ where $(p-\delta)' > p'.$
Restating, for each $w\in A_p$ there exists some $\varepsilon>0$ such that the conclusions \eqref{appB:thm0} and \eqref{appB:thm1} hold for all $1\leq r < p'+\varepsilon.$
\end{rem}

Both Theorems \ref{appBthm0} and \ref{appBthm1} follow almost immediately from the following
\begin{prop}\label{appBprop} Let $1\leq p <\infty$ and suppose that $w\in A_p,$ and let $\alpha\in[0,\infty).$  Then, for each $f\in L^1_{\loc}$ and a cube $Q_0,$ there exists a sparse collection $\mathscr{S}(Q_0)\subset \mathscr{D}(Q_0)$ 
	such that
	\begin{align}
		\calO_w^{r,\alpha}(f;Q_0)\lesssim \Big(\frac{1}{w(Q_0)^{1+\frac{\alpha}{d} r}}\sum_{Q\in\mathscr{S}(Q_0)}\calO_w^{1,\alpha}(f;Q)^rw(Q)^{1+\frac{\alpha}{d} r}\Big)^{\frac{1}{r}},
	\end{align}
	whenever $1\leq r \leq p'$ and $r<\infty.$ 
\end{prop}

\begin{proof}[Proof of Theorem \ref{appBthm0}] Suppose first that $b\in \BMO_w^{r,\alpha}$ for some $r>1.$ Then, by H\"older's inequality we find 
	\begin{align*}
		\int_Q\abs{b-\ave{b}_Q} &\leq 	\left(\int_Q\abs{b-\ave{b}_Q}^r w^{-\frac{r}{r'}}\right)^{\frac{1}{r}}\Big(\int_Q w\Big)^{\frac{1}{r'}} \\ 
		&=	\left( \frac{1}{w(Q)}\int_Q\big(\abs{b-\ave{b}_Q}\frac{1}{w}\big)^{r}\ud  w\right)^{\frac{1}{r}}\Big(\int_Q w\Big)^{\frac{1}{r}}\Big(\int_Q w\Big)^{\frac{1}{r'}} \\ 
		&\leq  \Norm{b}{\BMO_w^{r,\alpha}} w(Q)^{\alpha/d}w(Q).
	\end{align*}
For the other direction, by Proposition \eqref{appBprop} we estimate
\begin{align*}
	\calO_w^{r,\alpha}(f;Q_0)&\lesssim \Big(\frac{1}{w(Q_0)^{1+\frac{\alpha}{d} r}}\sum_{Q\in\mathscr{S}(Q_0)}\calO_w^{1,\alpha}(f;Q)^rw(Q)^{1+\frac{\alpha}{d} r}\Big)^{\frac{1}{r}},\\ 
	&\leq 	\Norm{b}{\BMO_w^{1,\alpha}}  \Big(\frac{1}{w(Q_0)}\sum_{Q\in\mathscr{S}(Q_0)}w(Q)\Big)^{\frac{1}{r}}\lesssim 	\Norm{b}{\BMO_w^{1,\alpha}},
\end{align*}
where we used sparseness in the last estimate; recall that sparseness with respect to any measure in $A_{\infty}$ is equivalent with sparseness with respect to the Lebesgue measure.
\end{proof}

\begin{proof}[Proof of Theorem \ref{appBthm1}] Suppose that $b\in \VMO_{w}^{1,\alpha}$ and we will show that $b\in\VMO_w^{r,\alpha}.$ Let us consider the different kinds of cubes as in the conditions \eqref{eq:VMOv1}, \eqref{eq:VMOv2}, \eqref{eq:VMOv3}. If a cube is sufficiently small, or sufficiently far away from the origin, as respectively in the conditions  \eqref{eq:VMOv1} and  \eqref{eq:VMOv3},  then all the subcubes are also, and it is immediate from Proposition \ref{appBprop} that the conditions  \eqref{eq:VMOv1} and  \eqref{eq:VMOv3} hold with all $1\leq r\leq p'.$ Let us then turn to check the condition  \eqref{eq:VMOv2}. 
	
By conditions \eqref{eq:VMOv1} and \eqref{eq:VMOv3} there exists a cube $P_0$ such that if $\ell(Q)\geq \ell(P_0)$ or $Q\cap P_0\not=\emptyset,$ then $\calO_w^{1,\alpha}(b;Q)\leq \varepsilon,$ and let $Q$ be exactly such a cube. We show that $\calO_w^{r,\alpha}(b;Q)\lesssim \varepsilon,$ provided that $Q$ is taken sufficiently large. Let $Q_0\supset Q$ be a cube such that $\ell(Q_0)\lesssim \ell(Q)$ and $P_0\in\mathscr{D}(Q_0),$ notice that we can always arrange this with the implicit constant independent of $Q.$ By $w$ being doubling, we know that $\calO_w^{r,\alpha}(b;Q)\lesssim 	\calO_w^{r,\alpha}(b;Q_0).$ 
Let $\mathscr{S}(Q_0)\subset \mathscr{D}(Q_0)$ be the sparse collection as in Proposition \ref{appBprop} and we estimate
\begin{align*}
	w(Q_0)^{1+\frac{\alpha}{d}r}\calO_w^{r,\alpha}(b;Q_0)^r\lesssim\Big( \sum_{\substack{L\in\mathscr{S}(Q_0) \\ L\subset P_0}} + \sum_{\substack{L\in\mathscr{S}(Q_0) \\ L\cap P_0^c\not=\emptyset}}\Big)\calO_w^{1,\alpha}(b;L)^rw(L)^{1+\frac{\alpha}{d} r}
\end{align*}
The first sum on the right is controlled by the fact that $\VMO_w^{1,\alpha}\subset \BMO_w^{1,\alpha}$ as
\begin{align*}
	\sum_{\substack{L\in\mathscr{S}(Q_0) \\ L\subset P_0}} \calO_w^{1,\alpha}(b;L)^rw(L)^{1+\frac{\alpha}{d} r}\leq \Norm{b}{\BMO_w^{1,\alpha}}^r	\sum_{\substack{L\in\mathscr{S}(Q_0) \\ L\subset P_0}} w(L)^{1+\frac{\alpha}{d} r}\lesssim \Norm{b}{\BMO_w^{1,\alpha}}^rw(P_0)^{1+\frac{\alpha}{d} r},
\end{align*}
where we used that $\Norm{\cdot}{\ell^s}\leq \Norm{\cdot}{\ell^1},$ for $s\geq 1,$ and sparsity;
while for the second sum if $L\cap P_0^c\not=\emptyset,$ then by $L,P_0\in\mathscr{D}(Q_0)$ either $P_0\subset L$ or $L\subset P_0^c,$ and in both cases by the choice of $P_0$ we know that $\calO_w^{1,\alpha}(b;L)\leq \varepsilon.$ 
Hence, we find that
\begin{align*}
	w(Q_0)^{1+\frac{\alpha}{d} r}\calO_w^{r,\alpha}(b;Q_0)^r\lesssim \Norm{b}{\BMO_w^{1,\alpha}}^r w(P_0)^{1+\frac{\alpha}{d} r}+ \varepsilon^r w(Q_0)^{1+\frac{\alpha}{d} r},
\end{align*}
from which the claim follows by choosing $Q$ (hence $Q_0$) sufficiently large.
\end{proof}

In proving Proposition \ref{appBprop} we can use the following almost orthogonality principle.
\begin{lem}\label{lem:Pyth} Let $1<p<\infty$ and $w\in A_p,$ let $\mathscr{S}\subset \mathscr{D}$ be a sparse collection of cubes. For each $Q\in\mathscr{S},$ let $f_Q$ be a function supported on $Q$ that is constant on each $P\in \ch(Q),$
where $\ch(Q)$ denotes the maximal elements of $\{P\in\mathscr{S}:P\subsetneq Q\}$.
Then, there holds that 
	\begin{align}
		\bNorm{\sum_{Q\in\mathscr{S}} f_Q }{L^p(w)}\lesssim \big(\sum_{Q\in\mathscr{S}}\Norm{f_Q}{L^p(w)}^p\big)^{\frac{1}{p}}.
	\end{align} 
\end{lem}
Lemma \ref{lem:Pyth} is certainly known but we could not find a reference, hence we recall it here.
\begin{proof}
By $f_Q$ being constant on $\ch(Q)$ we can write
\begin{align}
		\bNorm{\sum_{Q\in\mathscr{S}} f_Q }{L^p(w)}\leq 	\bNorm{\sum_{Q\in\mathscr{S}} f_Q1_{E_Q} }{L^p(w)} + 	\bNorm{\sum_{Q\in\mathscr{S}}\sum_{P\in\ch(Q)} \ave{f_Q}_P1_P}{L^p(w)}
\end{align}
and by disjointness of $E_Q\subset Q$ the left sum on the right clearly satisfies the desired estimate. For the second one, by duality, it is enough to estimate as follows. We have
\begin{align*}
	\int \sum_{Q\in\mathscr{S}}\sum_{P\in\ch(Q)} &\ave{f_Q}_P1_P\cdot g = \sum_{Q\in\mathscr{S}}\sum_{P\in\ch(Q)} \ave{f_Q}_P w(P)^{\frac{1}{p}} \ave{g}_P \abs{P} w(P)^{-\frac{1}{p}} \\
	&\leq \Big( \sum_{Q\in\mathscr{S}}\sum_{P\in\ch(Q)} \ave{\abs{f_Q}}_P^p w(P)\Big)^{\frac{1}{p}}\Big( \sum_{Q\in\mathscr{S}}\sum_{P\in\ch(Q)} \ave{\abs{g}}_P^{p'} \abs{P}^{p'}w(P)^{-\frac{p'}{p}}\Big)^{\frac{1}{p'}}.
\end{align*}
By $f_Q$ being constant on $\ch(Q)$ we have 
\begin{align}
	 \Big( \sum_{Q\in\mathscr{S}}\sum_{P\in\ch(Q)} \ave{\abs{f_Q}}_P^p w(P)\Big)^{\frac{1}{p}} \leq \Big(\sum_{Q\in\mathscr{S}}\Norm{f_Q}{L^p(w)}^p\Big)^{\frac{1}{p}}
\end{align}
and it remains to estimate the term with the function $g.$ By Hölder's inequality, we know that
$\abs{P}^{p'}w(P)^{-\frac{p'}{p}} = \abs{P}\ave{w}_P^{-\frac{p'}{p}} \leq w^{-\frac{p'}{p}}(P),$
and hence 
\begin{align*}
	 \Big( \sum_{Q\in\mathscr{S}}\sum_{P\in\ch(Q)} \ave{\abs{g}}_P^{p'} \abs{P}^{p'}w(P)^{-\frac{p'}{p}}\Big)^{\frac{1}{p'}}&\leq \Big( \sum_{Q\in\mathscr{S}}\sum_{P\in\ch(Q)} \ave{\abs{g}}_P^{p'} w^{-\frac{p'}{p}}(P)\Big)^{\frac{1}{p'}} \\
	 &\lesssim [w^{-\frac{p'}{p}}]_{A_{p'}}^{\frac{p}{p'}}\Norm{g}{L^{p'}(w^{-\frac{p'}{p}})} \sim_{[w]_{A_p}}  \Norm{g}{L^{p'}(w^{-\frac{p'}{p}})},
\end{align*}
where in the second estimate we used Lemma \ref{lem:prePyth} below.
\end{proof}

\begin{lem}\label{lem:prePyth} Suppose that $1<p<\infty,$ that $w\in A_p$ and that $\mathscr{S}$ is sparse. Then, there holds that 
	\begin{align*}
		\Big(\sum_{Q\in\mathscr{S}} \ave{f}_Q^pw(Q)\Big)^{\frac{1}{p}}\lesssim \Norm{f}{L^p(w)}.
	\end{align*}
\end{lem}
\begin{proof} By sparseness and the boundedness of the maximal function on $L^p(w)$, we find that
\begin{align*}
		\Big(\sum_{Q\in\mathscr{S}} \ave{f}_Q^pw(Q)\Big)^{\frac{1}{p}} \sim_{[w]_{A_p}} \Big(\int\big(\sum_{Q\in\mathscr{S}} \ave{f}_Q^p 1_{E_Q}\big)\ud w\Big)^{\frac{1}{p}} \leq \Big(\int (Mf)^p\ud w\Big)^{\frac{1}{p}} \lesssim_{[w]_{A_p}} \Norm{f}{L^p(w)},
\end{align*}
where we used the boundedness of the maximal operator in the last step.
\end{proof}

\begin{proof}[Proof of Proposition \ref{appBprop}] Given that $f\in L^1_{\loc},$ by an iterated Calderón-Zygmund decomposition, see e.g. \cite{LOR1} and this is also a special case of Lemma \ref{lem:aug}, there exists a sparse collection $\mathscr{S}(Q_0)\subset\mathscr{D}(Q_0)$ so that 
	\begin{align*}
		\abs{f-\ave{f}_{Q_0}}1_{Q_0}\lesssim \sum_{Q\in\mathscr{S}(Q_0)}\ave{\abs{f-\ave{f}_Q}}_Q1_Q.
	\end{align*}
Note that $p\leq r'$ (by $r\leq p'$) and hence $w^{1-r}\in A_r$ (by $w\in A_p\subset A_{r'}$). Hence, by  Lemma \ref{lem:Pyth} we estimate
\begin{align*}
	w(Q_0)^{1+\frac{\alpha}{d} r}\calO_w^{r,\alpha}(f;Q_0)^r &\lesssim 	\int_{Q_ 0}\Big(\sum_{Q\in\mathscr{S}(Q_ 0)} \bave{\abs{f-\ave{f}_Q}}_Q \frac{1_Q}{w}\Big)^r\ud w \\ 
	&= \BNorm{\sum_{Q\in\mathscr{S}(Q_0)} w(Q)^{\frac{\alpha}{d}}\ave{w}_Q\calO_w^{1,\alpha}(f;Q)1_Q}{L^r(w^{1-r})}^r \\ 
	&\lesssim [w^{1-r}]_{A_r}^r\sum_{Q\in\mathscr{S}(Q_0)}\BNorm{ w(Q)^{\frac{\alpha}{d}}\ave{w}_Q\calO_w^{1,\alpha}(f;Q)1_Q}{L^r(w^{1-r})}^r \\ 
	&\sim_{[w]_{A_{r'}}}\sum_{Q\in\mathscr{S}(Q_0)}  w(Q)^{\frac{\alpha}{d} r}\calO_w^{1,\alpha}(f;Q)^r\ave{w}_Q^rw^{1-r}(Q) \\ 
	&\lesssim_{[w]_{A_{r'}}}\sum_{Q\in\mathscr{S}(Q_0)}\calO_w^{1,\alpha}(f;Q)^r w(Q)^{1+\frac{\alpha}{d} r},
\end{align*}
	where in the last estimate we used that $\ave{w}_Q^rw^{1-r}(Q)\leq [w]_{A_{r'}}^{\frac{1}{r'-1}}w(Q).$ Now the proof is concluded after dividing with $w(Q_0)^{1+\frac{\alpha}{d} r}$ and taking  the $r$th root. 
\end{proof}

\subsection*{Acknowledgements} We would like to thank the anonymous referee for constructive comments that improved the presentation.

\bibliography{references}
\bibliographystyle{abbr}

\end{document}